\documentclass[11pt,a4paper,twoside] {extarticle}
\usepackage{amsthm,amssymb,amsmath,amsfonts,amscd}
\usepackage{tocloft}
\usepackage{mathtools}
\usepackage{tikz}
\usepackage{tikz-cd}
\usetikzlibrary{cd}
\usepackage{hyperref}
\usepackage{graphicx}
\usepackage[all]{xy}
\usepackage{color,soul}
\theoremstyle{plain}
\usepackage[margin=  1 in]{geometry}
\usepackage [english]{babel}
\usepackage[autostyle]{csquotes}
\usepackage{dsfont}
\usepackage{titlesec}
\usepackage{enumitem}   
\usepackage{stmaryrd}
\usepackage{setspace}
\usepackage{mathrsfs}

\newcommand{\gl}{\operatorname{GL}}
\newcommand{\SL}{\operatorname{SL}}
\renewcommand{\varprojlim}{\underleftarrow{\text{lim}}}
\newcommand{\Qbar}{\overline{\mathbb{Q}}}
\newcommand{\Q}{\mathbb{Q}}
\newcommand{\Z}{\mathbb{Z}}
\newcommand{\C}{\mathbb{C}}

\newcommand{\Zp}{\mathbb{Z}_{p}}

\newtheorem*{notation*}{Notation}
\newtheorem{theorem}{Theorem}[section]
\newtheorem*{theorem*}{Theorem}
\newtheorem{corollary}[theorem]{Corollary}
\newtheorem{lemma}[theorem]{Lemma}
\newtheorem{remark}[theorem]{Remark}
\newtheorem{proposition}[theorem]{Proposition}

\newtheorem{definition}[theorem]{Definition}

\theoremstyle{definition}
\newtheorem{example}[theorem]{Example}

\numberwithin{equation}{section}		
\numberwithin{figure}{section}			
\numberwithin{table}{section}				

\title{
	\usefont{OT1}{bch}{b}{n}
	\huge $p$-adic Rankin product $L$-functions 
}

\author{Eknath Ghate and Ravitheja Vangala}

\date{}

\begin{document}
\maketitle

\begin{abstract}
      We describe Panchishkin's construction of the $p$-adic Rankin product $L$-function.
\end{abstract}

\vspace{.2cm}

Let $p$ be an odd prime. In this article we give a construction of the  $p$-adic Rankin product 
$L$-function which interpolates $p$-adically  the special values of the convolution of two cusp forms 
on the  complex upper half plane. The argument given here closely follows Panchishkin's original argument
\cite{Panrankin} where the $S$-adic non-archimedean $L$-function associated to the Rankin product of two modular 
forms was constructed, for $S$ any set of finite primes including $p$. In this exposition we will
specialize the argument given in \cite{Panrankin} to the case $S= \lbrace p \rbrace$. We also provide some background
details and correct a sign error along the way which does not seem to have been noticed in the subsequent literature. 

\section{Introduction}

\subsection{Rankin product $L$-functions}

Let $N$ be an arbitrary natural number. We consider a cusp form $f$ of weight $ k \geq 2 $ for the 
congruence subgroup $\Gamma_{0}(N)$ and nebentypus $\psi$.  We suppose that $ f $ is a 
primitive cusp form, i.e., it is a normalized newform of some level $ C_{f} $ dividing $ N $; $ C_{f} $ 
is called the conductor of  $f$. Let $ g $ be another primitive cusp form of conductor $ C_{g} $ and 
weight $2 \leq l < k$ for  $ \Gamma_{0} (N) $ and nebentypus $\omega$. We set $e(z)=e^{2\pi i z}$ and let
\begin{align}
      f(z)=\sum\limits_{n=1}^{\infty} a(n)e(nz), \ \ \  
      g(z)=\sum\limits_{n=1}^{\infty}b(n)e(nz)
\end{align}
be the Fourier expansions of $f$ and $g$.
The Rankin convolution of the modular forms $ f $ and $ g $ is defined by means of the equality
\begin{align}\label{1.2}
       \mathcal{D}(s,f,g) := L_{N}(2s+2-k-l,\psi\omega) L(s,f,g),
\end{align}
where
$$
	   L(s,f,g) = \sum\limits_{n=1}^{\infty} a(n)b(n)n^{-s},
$$
and $ L_{N}(2s+2-k-l,\psi\omega) $ denotes the Dirichlet $ L $-series with character $\psi\omega$, 
and the subscript $N$ indicates that the factors corresponding to the prime divisors of $N$ are
omitted from the Euler product. A classical method of Rankin and Selberg  \cite{RankinSelberg} 
enables one to construct an analytic continuation of the function $ \mathcal{D}(s,f,g) $ to the 
whole  complex plane and prove that it satisfies a functional equation. Let
$$
f^{\rho}(z):= \ \sum\limits_{n=1}^{\infty} \overline{a(n)} e(nz), \ \ \   g^{\rho}(z):= \ \sum\limits_{n=1}^{\infty} 
\overline{b(n)} e(nz).
$$
Further, define
\begin{align}\label{1.4}
       \Psi(s,f,g)&=\gamma(s)\mathcal{D}(s,f,g), 
\end{align}
where $\gamma(s)=(2\pi)^{-2s}\Gamma(s)\Gamma(s-l+1)$ consists of $\Gamma$-functions. Though we 
do not use it here, $\Psi(s,f,g)$ has a well-known functional equation. For instance, if $\psi$, $\omega$ and 
$\psi^{-1}\omega$ all have conductor $N$ and $C_{f} = C_{g} = N$, then the functional 
equation is (see \cite[$\S$9.5, Theorem 1]{hida_93}):
\begin{align}
      \Psi(s,f^{\rho},g)= W(f^\rho,g) N^{3(-s+(k+l-1)/2)} \Psi(k+l-1-s,f,g^{\rho}),
\end{align}
where
\begin{align}
       W(f^\rho,g) = (-1)^{l}  \Lambda(f^{\rho})  \Lambda(g) \frac{G(\psi^{-1} \omega)}{\lvert G(\psi^{-1} \omega)
       \rvert}, \nonumber
\end{align}
$G(\psi^{-1}\omega)$ is the Gauss sum associated to $\psi^{-1}  \omega$ and $\Lambda(f^\rho)$, $\Lambda(g)$ 
are the root numbers associated to $f^\rho$, $g$ respectively (defined in $\S$2).  Shimura \cite{Shi_76} established  
the following algebraicity result for the special values of 
$\mathcal{D}(s,f,g)$ (see \cite[$\S$10.2, Theorem 1]{hida_93}): the numbers
\begin{align}\label{1.5}
       \Psi(s,f,g)(\pi^{1-l}\langle f,f \rangle_{C_f})^{-1} \in \mathbb{\overline{Q}},
\end{align}
for all integers $l \leq s \leq k-1$. Here $\langle f,f \rangle_{C_{f}}$ is the 
Petersson inner product defined by
$$
        \langle f,f	\rangle_{C_{f}} := \int_{\mathcal{H}/\Gamma_{0}(C_{f})} 
        \vert f(z) \vert ^{2} y^{k-2} \ dx  dy, \ \ \  z=x+iy,
$$
where $ \mathcal{H}/\Gamma_{0}(C_{f})$ is a fundamental domain for the upper half plane 
$\mathcal{H}$ modulo the action of $\Gamma_{0}(C_{f})$. The integers $s = l, \ldots , k-1$ in 
\eqref{1.5} are \enquote{critical} in the sense of Deligne \cite{Del79}. They are precisely the 
values of $s$ for which neither of the functions  $\gamma(s)$ and $ \gamma(k+l-1-s)$ 
in the functional equation have poles.

\subsection{Main theorem}

Let  $\mathbb{C}_{p}=\widehat{\overline{\mathbb{Q}}}_p$ be the completion of  the algebraic closure of  
$\mathbb{Q}_{p}$. Let $\vert \cdot \vert_p$ be the norm on $\C_p$, normalized so that $\vert p \vert_p = 1/p$. 
For any topological group $G$, let $ X({G})$ denote the group of continuous homomorphisms 
from  $G$ to $\mathbb{C}_{p}^{\times}$. The domain of definition of  $p$-adic $L$-functions is the 
$\mathbb{C}_{p}$-analytic Lie group $X_{p} = X(\mathbb{Z}_{p}^{\times})$, where $\mathbb{Z}_{p}^{\times}$ is 
the group of units of $\mathbb{Z}_{p}$. We put $X_{p}^{\mathrm{tors}}= \lbrace \chi \in X_{p} \mid \chi 
\mathrm{ \ has \ finite \ order} \rbrace$. Let  $x_{p}$ denote the embedding $\mathbb{Z}_{p} \hookrightarrow 
\mathbb{C}_{p}$. For a precise statement of the results we introduce the notation    
   $$
           g(\chi) := \ \sum\limits_{n=1}^{\infty} \chi(n)b(n) e(nz),
   $$
for the cusp form $g$ twisted by the Dirichlet character $\chi$. 
We fix an  embedding of ${\overline {\mathbb Q}}$ into ${\mathbb{C}}$ and an embedding 
$i_{p}: \mathbb{\overline{Q}} \hookrightarrow \mathbb{C}_{p}$.  
Then every Dirichlet character $\chi$ whose conductor  $C_\chi$ is a power of  $p$  
can be identified with an element of $X_{p}^\mathrm{tors}$ and  vice versa. 
By Theorem~\ref{Algebraicity}, with $g(\chi)$ replaced by $g^\rho(\overline{\chi})$, the numbers 
$$
\frac{\Psi(l+r,f,g^{\rho}(\overline{\chi}))}{\pi^{1-l}\langle f,f \rangle_{C_f}} \in \overline{\mathbb{Q}},
$$
for  $r = 0,1, \ldots,k-l-1$. 
  In this article we construct  a $\mathbb{C}_{p}$-analytic function on  $X_{p}$ which 
  interpolates the numbers
\begin{align*}
        i_{p} \left(\frac{\Psi(l+r,f,g^{\rho}(\overline{\chi}))}{\pi^{1-l}\langle f,f \rangle_{C_f}} \right),
\end{align*}
for all  $\chi \in X_{p}^{\mathrm{tors}}$ and $r = 0, 1, \ldots, k-l-1$. We work under the assumption 
that $f$ is a $p$-ordinary form, i.e., $a(p)$  is a unit in  $\mathbb{C}_{p}$. In other words
\begin{align}\label{1.9}
      \vert i_{p}(a(p))\vert_{p} = 1.
\end{align}
In addition, we suppose that 
\begin{align}\label{1.10}
        (C_{f},C_{g})=1, \ \ (p,C_{f})=(p,C_{g}) =1,     
\end{align}
and we set $C=C_{f}C_{g}$. Let $\alpha(p)$ denote the root  of the  Hecke polynomial  
$X^{2}-a(p)X+\psi(p)p^{k-1}$, for which $\vert i_{p}(\alpha(p))\vert_{p}=1$ and 
let $\alpha'(p)$ be the other root.  For every prime $q\nmid N$, let
\begin{equation}\label{alpha,beta}
       \begin{split}
 	      X^{2} -a(q)X + \psi(q)q^{k-1}  & = (X-\alpha(q)) (X-\alpha'(q)), \\
              X^{2}- b(q)X + \omega(q)q^{l-1}& = (X-\beta(q))(X-\beta'(q)).
      \end{split}
 \end{equation}
We extend the definition of $\alpha(n)$  to all natural numbers of the form $p^{r}$ 
by setting  $\alpha(p^{r}) := \alpha(p)^{r}$. 
 
\begin{theorem} \label{maintheorem} \emph{(}\textbf{Main theorem}\emph{)} 
 	Under the assumptions \eqref{1.9} and \eqref{1.10}, there  exists  a unique measure $\mu$ on 
        $\mathbb{Z}_{p}^{\times}$ satisfying  the following interpolation property: for all characters 
         $\chi \in X_{p}^{\mathrm{ tors}}$ and all integers $r$ with $0 \leq r \leq k-l-1$, the value of
          the function $x_p^r \chi$ under the measure $\mu$
         $$
           \int_{\mathbb{Z}_{p}^{\times}} x_{p}^{r} \chi  \  d\mu
         $$ is given by the image under $i_{p}$ of the 
         following algebraic number 
         \begin{align*}
 	       (-1)^{r} \omega(C_{\chi})\frac{G(\chi)^{2}C_{\chi}^{l+2r-1}}
 	        {\alpha(C_{\chi}^{2})}   \cdot  \frac{\Psi(l+r,f,g^{\rho} 
 	  	(\overline{\chi}))}{\pi^{1-l} \langle f , f \rangle_{C_{f}}} .
         \end{align*}
\end{theorem} 

This is exactly\footnote{Except that we have added the sign $(-1)^r$ which we feel is necessary 
(see subsequent footnotes).} \cite[Thm. 1.4]{Panrankin}, noting that the extra Euler factors 
$A(r,\chi)$ there do not appear here because $S = \{ p \}$. 
%
%
Finally we remark that if $\mu$ is a 
$\mathbb{C}_{p}$-valued measure on $\mathbb{Z}_{p}^{\times}$, as in the main theorem above, 
then the function $L_\mu$ (the $p$-adic $L$-function attached to $\mu$) 
defined by
\begin{align}
       L_{\mu}(\chi) =   \mu(\chi) = \ \int_{\mathbb{Z}_{p}^{\times}} \chi 
         \ d\mu, \ \ \forall \ \chi \in X_{p} ,
\end{align}
always turns out to be a 
$\mathbb{C}_{p}$-analytic function $L_{\mu}:X_{p}\longrightarrow 
\mathbb{C}_{p}$. 

To make sense of the last statement we briefly recall the $\mathbb{C}_{p}$-analytic structure on 
$X_p = X(\mathbb{Z}_{p}^{\times})$.  
We set
$$
       U = \lbrace x \in \mathbb{Z}_{p}^{\times} 
        \mid x \equiv 1 \> \textrm{mod} \  p \rbrace,
$$
 units of $\mathbb{Z}_{p}$ congruent to 1 mod $p$. Then we have the following decomposition
$$
        X_{p} = X (( \mathbb{Z}  / p\mathbb{Z})^{\times})  \times X(U).
$$
Therefore every $\chi \in X_{p}$ can be written as
$\chi_{0} \chi_{1}$ with $\chi_{0} \in  X (( \mathbb{Z}  / p \mathbb{Z}) ^{\times})$ and $\chi_{1} \in X(U)$. 
The characters $\chi_{0}$ and $\chi_{1}$ are called the tame part and the wild part of the character $\chi$ 
respectively. 

We claim the function $\varphi$ defined by $\varphi(\chi) := \chi(1+p)$, where $1+p$ is a topological 
generator of the 
group $U$, induces an isomorphism of groups

$$
    \varphi: X(U) \stackrel{\sim}{\longrightarrow} T :=  \lbrace t \in \mathbb{C}_{p}^{\times} \mid  \vert t -1 
    \vert_{p}  < 1 \rbrace.
$$
This isomorphism defines an analytic structure on $X(U)$, which can easily be checked to be independent 
of the choice of generator $1+p$. 
We first check that $\varphi$ is well defined, i.e., $\varphi$ takes values in $T$. Let $\chi \in X(U)$. 
Since $(1+p)^{p^{n}} \rightarrow 1 $ as $n \rightarrow \infty $, by the continuity of $\chi$ we have 
$(\chi(1+p))^{p^{n}} \rightarrow 1$.  Hence, $ \vert \chi (1+p) \vert_{p} = 1 $ and  
$\lvert \chi(1+p) - 1 \rvert_{p} \leq\mathrm{max} \lbrace \lvert \chi(1+p) \rvert _{p}, 1 \rbrace \leq 1$. 
We now claim that  $\lvert \chi(1+p) -1 \rvert_{p} < 1$. Suppose not, then $\lvert \chi(1+p)-1 \rvert_{p} =1$. 
Therefore
\begin{align*}
         1 & =  \vert (\chi(1+p)-1)^{p^{n}} \vert_{p}  \\
         &=   \Big\vert \sum\limits_{k=1}^{p^{n}} \binom{ \ p^{n} }{k} (\chi(1+p)-1)^{k}  \Big\vert_{p} 
         \ \ \ \  (\mathrm{ as} \  p  \Bigm\vert  \binom{ p^{n}}{k \ }  (\chi(1+p)-1)^{k},  \ \forall \  1 \leq  k  < p^{n})\\   
         &=  \lvert (\chi(1+p)-1+1)^{p^{n}} - 1 \rvert _{p} \\
         & =  \lvert (\chi(1+p))^{p^{n}} - 1 \rvert _{p} .
\end{align*}
But, this contradicts $(\chi(1+p))^{p^{n}}  \rightarrow 1$. Hence,  $\vert \chi(1+p)-1 \vert_{p} < 1$.  
We now show that $\varphi$ is an isomorphism. Every character $\chi \in X_{p}$ is uniquely determined by 
$\chi(1+p)$, since  $1+p$ is a topological generator of $U$, hence $\varphi$ is injective. 
For $t \in T$, define  $\chi_{t}((1+p)^{n}) =t^{n}$, for all $n \in \mathbb{Z}$ . Extending  $\chi_{t}$ to 
all of $1+p\mathbb{Z}_{p}$ by continuity we get an element of $X(U)$ which maps to $t$ under $\varphi$. 
Hence $\varphi$ is also surjective. 
A function $F: T \rightarrow \mathbb{C}_{p}$ is said to be  analytic if $F(t)$ can be expressed as a
power series, i.e.,  $F(t)=\sum_{i=0}^{\infty} a_{i}(t-1)^{i}, \ a_{i} \in \mathbb{C}_{p}$, which converges 
absolutely for all $t \in T$. The isomorphism  $\varphi : X(U) \simeq T$ allows us to define an analytic structure
 on $X(U)$.  Finally the notion of analyticity can be extended to all of $X_{p}$ by translation. 


In closing this introduction, we remark that Hida \cite{Hid88} subsequently 
constructed a more general measure interpolating the critical Rankin product 
$L$-values of two cusp forms which themselves vary in $p$-adic families,
\footnote{The sign mentioned in the previous footnote is consistent with the sign in 
 \cite[Theorem I]{Hid88}.} and in a different direction, Vienney \cite{Vie00}
has generalized Panchishkin’s argument to cases where $a(p)$ is not a $p$-adic unit.
\footnote{Again, the author adds a sign of $(-1)^r$.}

\subsection{Outline  of the paper} 

We recall notation and results from the theory of modular forms in 
$\S2$. In $\S$3 we recall  generalities about distributions and measures and state a 
criterion for a distribution whose values are known on a specific set of functions to be a measure
 in terms of the abstract Kummer congruences. The measure  in Theorem~\ref{maintheorem}  is 
 obtained from certain complex-valued distributions $\Psi_{s}$, which we construct in  $\S$4 using 
 the definition of the convolution \eqref{1.2}. The distributions $\Psi_{s}$  take values 
 in $\Qbar$ on $X_{p}^{\mathrm{tors}}$ for integers $l \leq s \leq k-1$. In $\S$5 we obtain an integral representation for these 
 distribution values using the Rankin-Selberg method  and holomorphic projection. 
In $\S$6  we prove that the ${\mathbb{C}}_p$-valued distributions  $i_{p}(\Psi_{s})$ satisfy the 
abstract Kummer congruences to finish the proof of Theorem \ref{maintheorem}. 

\section{Background on Modular forms}

In this section we  recall a few results from  the theory of  classical modular forms.  
Most of the material covered here is well known. In this section  $f$ and $g$  
are arbitrary functions which need not satisfy the assumptions of $\S1$ unless  otherwise
 stated. Also, $\chi$, $\psi$, $\omega$ denote Dirichlet characters. 
Let $\mathrm{M}_{2}(\mathbb{Z})$ denote the set of $2\times2$ matrices with entries in 
$\mathbb{Z}$. Let $\SL_{2}(\mathbb{Z})$ denote the set of matrices with determinant 1 
in $\mathrm{M}_{2}(\mathbb{Z})$.   Let $\mathbb{C}$ denote the complex plane. We write an 
element $z \in \mathbb{C}$ as $x+iy$, where $x,y \in \mathbb{R}$ and $i^{2}=-1$. For 
$z =x+iy \in \mathbb{C} $, sometimes we denote $x$ and $y$ by Re($z$) and Im($z$) 
respectively. 

\subsection{Classical modular forms}
Let  $\mathcal{H}= \lbrace z  \in \mathbb{C} \mid \mathrm{Im}(z)> 0 \rbrace $ denote the complex 
upper half plane, on which the group $\gl_{2}^{+}(\mathbb{R})$ of real  $2 \times 2$ matrices with 
positive determinant acts by fractional linear transformations. For any natural number $k$, we have a weight $k$ 
action of $\gl_{2}^{+}(\mathbb{R})$ on functions $f: \mathcal{H} \rightarrow \mathbb{C}$ 
given by:
\begin{equation*}
      (f\vert_{k}\gamma)(z) =( \mathrm{det}\gamma)^{k/2} (cz+d)^{-k} f\bigg (\frac{az+b}{cz+d}\bigg ), 
       \ \ \forall \ \gamma =   \begin{pmatrix}  a & b \\ c & d \end{pmatrix} \in  \gl_{2}^{+}(\mathbb{R}).
\end{equation*}
For any natural number $N$, we have the following subgroups:
\begin{eqnarray*}
       \Gamma_{0}(N)  & =& \bigg\lbrace  \begin{pmatrix} a & b \\ c & d  \end{pmatrix}
         \in \SL_{2}(\mathbb{Z}) \mid c \equiv 0 \  \ \text{mod } N \bigg\rbrace, \\
       \Gamma_{1}(N) & =& \bigg\lbrace \begin{pmatrix} a & b \\ c & d \end{pmatrix}
      \in \Gamma_{0}(N) \mid a \equiv d \equiv 1 \  \ \text{mod } N \bigg\rbrace, \\
      \Gamma(N) & =& \bigg\lbrace \begin{pmatrix} a & b \\ c & d  \end{pmatrix}
        \in \Gamma_{1}(N) \mid b \equiv 0 \  \ \text{mod } N \bigg\rbrace.
\end{eqnarray*}
\begin{definition}
       A subgroup $\Gamma$ of $\SL_{2}(\mathbb{Z})$ is  called a congruence subgroup if 
       $\Gamma(N)\subset\Gamma$ for some $N>0$. The smallest $N$ satisfying this condition 
       is called  the level of the congruence subgroup.
\end{definition}
If $\Gamma$ is a congruence group, then $M_{k}(\Gamma)$ denotes the  complex vector space of 
modular forms of weight $k$ for $\Gamma$.  These consist of holomorphic functions 
$f : \mathcal{H} \rightarrow \mathbb{C}$ which satisfy $f \vert_k \gamma = f$, for all 
$\gamma \in \Gamma$, and a  holomorphicity condition at the cusps of $\Gamma$. Let
$S_{k}(\Gamma)$ denote the subspace of cusp forms consisting of those $f$ which in addition 
vanish at the cusps.
 
 \begin{notation*} Throughout the article we use the following notation:
 	\begin{enumerate} [label=\emph{(\roman*)}]
 		\item For every integer $M$, let $S(M)$ denote the set of primes dividing $M$.
 		\item For every $\gamma= \begin{psmallmatrix} a & b \\ c & d \end{psmallmatrix} 
 		          \in  \mathrm{M}_{2}(\mathbb{Z})$, and Dirichlet character $\psi$ we put 
 		          $\psi(\gamma)= \psi(d)$.
 		\item Let $\chi_{0}$ denote the principal character on $\mathbb{Z}$. It is given by 
 		          $\chi_{0}(n)=1, \ \forall\  n \in \mathbb{Z}$. 
 	\end{enumerate}
\end{notation*}
\noindent If $\psi$  is a Dirichlet character mod $N$, we set
\begin{eqnarray*}
       M_{k}(N,\psi) & = & \lbrace f \in M_{k}(\Gamma_{1}(N)) \> \mid \> f\vert_{k}  
                           \gamma = \psi(\gamma) f, \ \forall \ \gamma  \in \Gamma_{0}(N) \rbrace, \\
       S_{k}(N,\psi) & = & S_{k}(\Gamma_{1}(N)) \cap M_{k}(N,\psi).
\end{eqnarray*}	     
For an arbitrary modular form $h \in M_{k}(N,\psi)$ with $k \geq 1$ and a cusp form 
$f \in S_{k}(N, \psi)$ the Petersson inner product is defined by  the integral 
\begin{equation}
         \langle f,h \rangle_{N}= \int_{\mathcal{H}/\Gamma_{0}(N)} \overline{f(z)}  h(z)  y^{k-2} \ dx  dy, 
\end{equation}
where $ \mathcal{H}/\Gamma_{0}(N)$ is a fundamental domain for the  upper half plane 
$\mathcal{H}$  modulo the action of $\Gamma_{0}(N)$.  Observe that if 
$\gamma \in \gl_{2}^{+}(\mathbb{R})$ normalizes $\Gamma_{0}(N)$ and $\gamma^2$ is a scalar matrix, then 
(see \cite[Theorem 2.8.2]{Miyake})
\begin{align}\label{2.2}
        \langle f \vert_{k} \gamma,h \rangle_{N}= \langle f  ,h\vert_{k}\gamma \rangle_{N}.
\end{align}
For the rest of this subsection assume that $M$, $N$ are positive integers such that $S(NM)=S(N)$. 
 Since $S(NM) = S(N)$ it can be checked that 
$[ \Gamma_{0}(N):\Gamma_{0}(NM)]=M$ and 
$$
\lbrace \beta_{u}  = \begin{psmallmatrix}1 & 0 \\ uN & 1 \end{psmallmatrix} \mid   u = 1, \ldots, M  \rbrace
$$ 
is a set of coset representatives for $\Gamma_{0}(NM) \backslash \Gamma_{0}(N)$. 
Therefore, for every $\gamma\in \Gamma_{0}(N)$ and $\beta_{u}$, there exists unique
$\gamma_{u} \in \Gamma_{0}(NM)$ and $\beta_{u'}$ such that $\beta_{u} \gamma = \gamma_{u} \beta_{u'}$. 
Since  $\beta_{u}$, $\beta_{u'} \equiv \mathrm{I}_{2}$ mod $N$ we have  
$$
      \gamma \equiv \beta_{u} \gamma  = \gamma_{u} \beta_{u'} \equiv \gamma_{u} \ \mathrm{mod} \ N.
$$
 For a Dirichlet character $\psi$ modulo $N$ and $h \in M_{k}(NM,\psi)$ we have 
$$
\left ( \sum_{u=1}^{M} h\vert_{k} \beta_{u} \right ) \vert_{k}\gamma = 	\sum\limits_{u= 1}^{M} h\vert_{k}  \gamma_{u}	\beta_{u'}	
    = \sum\limits_{u = 1}^{M}  \psi(\gamma_{u}) \cdot h\vert_{k} \beta_{u'} =\psi(\gamma) \cdot \sum\limits_{u = 1}^{M} h\vert_{k} \beta_{u}.$$
Therefore  $\sum_{u =1}^{M} h\vert_{k} \beta_{u} \in M_{k}(N,\psi)$. 
For $M$, $N$  positive integers such that $S(NM) = S(N)$ and  a Dirichlet character $\psi$ modulo $N$, 
define the trace operator $Tr_{N}^{NM}: M_{k}(NM,\psi) \rightarrow M_{k}(N,\psi)$  by the equality 
\begin{equation}
 	  Tr_{N}^{NM}(h) =  \sum\limits_{u=1}^{M} h \vert_{k} \beta_{u} 
 	  =  \sum\limits_{u=1}^{M} h \vert_{k} \left (\begin{smallmatrix}1 & 0 \\ uN & 1 \end{smallmatrix} \right).
\end{equation}

\begin{remark}
 The definition of the trace above depends on the choice of coset representatives 
 $ \lbrace \beta_1, \ldots, \beta_M \rbrace$ of $\Gamma_{0}(NM) \backslash \Gamma_{0}(N)$. 
 In the computations below, we always use this choice.
\end{remark}
\begin{lemma}\label{lemma2.2}
	Let $\psi$ be a Dirichlet character modulo $N$. Let  $f \in S_{k}(N,\psi)$ and $h \in M_{k}(NM,\psi) $. 
	If $S(M) \subset S(N)$, then $\langle f , h \rangle _{NM}= \langle f, Tr_{N}^{NM}(h)\rangle_{N}$. 
\end{lemma}	
\begin{proof}
       Let $\lbrace \beta_{1}, \ldots, \beta_{M} \rbrace $  be as above. If $\mathcal{D}$ is a fundamental 
       domain for $ \Gamma_{0}(N)$, then $\coprod_{u=1}^{M} \beta_{u} \mathcal{D}$ is a fundamental domain 
       for $\Gamma_{0}(NM)$. Therefore
       \begin{align*}
              \langle f,h \rangle_{NM} 
               &= \int_{\mathcal{H}/\Gamma_{0}(NM)} \overline{f(z)} h(z) y^{k-2} \ dx dy\\
	       & = \sum\limits_{u = 1}^{M}  \int_{\beta_{u} \mathcal{D}} \overline{f(z)} h(z) y^{k-2} \ dx   
	           dy  \\
               & = \sum\limits_{u =1}^{M}  \int_{\mathcal{D}} \overline{(f\vert_{k}\beta_{u})(z)}      
                   (h\vert_{k}\beta_{u})(z) y^{k-2} \ dx  dy \\
	       & = \sum\limits_{u =1}^{M} \int_{\mathcal{D}} \overline{f(z)}\ (h\vert_{k}\beta_{u})(z) 
	           y^{k-2} \ dx  dy \\
               & = \langle f, Tr_{N}^{NM}(h) \rangle_{N}. 
                \tag*{\qedhere}
       \end{align*} 
\end{proof}	
For any integer $k$, complex number $s$ and Dirichlet characters $\chi$, $\psi$  modulo $L$, $M$ 
respectively, we define (see \cite[Chapter 7]{Miyake}) the non-holomorphic Eisenstein series of weight $k$ by
\begin{equation}\label{Eisenstein series}
       E_{k}(z,s; \chi, \psi) = y^{s} \sideset{}{'} \sum\limits_{c,d \> =- \infty} ^ 
       {\infty}  \chi(c) \psi(d) (cz+d)^{-k} \vert cz+d \vert ^{-2s}, \ \forall \ z \in \mathcal{H},
\end{equation}
where the prime means that the sum  is over all $(c,d) \in \mathbb{Z}^{2} \smallsetminus 
\lbrace (0,0) \rbrace $. The series converges for Re$(k+2s) > 2$ and can be  continued 
meromorphically to the whole complex plane as a function of $s$. Further,  if $k\geq 3$, 
then  $E_{k}(z,0;\chi,\chi_{0}) \in M_{k}(L,\chi)$ \cite[Lemma 7.1.4, Lemma 7.1.5]{Miyake}.  

\subsection{Operators acting on  modular forms }
\label{operators}

 Let $f \in S_{k}(N,\psi)$ be a cusp form with the Fourier expansion 
$$
	    f(z)= \sum\limits_{n=1}^{\infty}a(n)e(nz).
$$
If $d$ is a natural number, then define
\begin{eqnarray*}
	    f \vert  U_{d} & = & \sum\limits_{n = 1}^{\infty} a(dn)e(nz) 
	     \> = \> d^{k/2-1} \sum\limits_{u \ \mathrm{mod} \ d} f\vert_{k}	 
         \begin{pmatrix} 1 &  u  \\ 0 & d \end{pmatrix}, \\
         f \vert  V_{d} & = & f(dz) \> = \> d^{-k/2} f\vert_{k} 
         \begin{pmatrix}  d & 0 \\ 0 & 1  \end{pmatrix} \ \in \ S_{k}(Nd, \psi),\\
	    f^{\rho}(z)   & = & \overline{f(-\overline{z})}
	     \> = \> \sum\limits_{n=1}^{\infty}\overline{a(n)} e(nz) \ \in \ S_{k}(N,\overline{\psi}),  \\ 
	    f\vert w_{d}   & = & (\sqrt{d}z)^{-k} f \Big ( \frac{-1}{dz} \Big ) \> = \> f\vert_{k} 
         \begin{pmatrix}  0 & -1 \\   d & 0  \end{pmatrix}, \ f\vert w_{N} \ \in \ S_{k}(N,\overline{\psi}).
\end{eqnarray*}
Also, the Hecke operators $T_{n}: M_{k}(N, \psi) \rightarrow M_{k}(N, \psi)$  are  defined by 
$(T_{n}f)(z)= \sum\limits_{m=0}^{\infty} a(m,T_{n}f)e(mz)$, where 
$a(m,T_{n}f)= \sum\limits_{0< d \mid (m,n)}\psi(d) d^{k-1}a(mn/d^{2})$. 

When $S(NM)=S(N)$, we have the following identity, which will be used for explicit computations:
\begin{eqnarray}\label{2.10}
        Tr_{N}^{NM}(f) & = & (-1)^{k} M^{1-k/2} \ f \vert w_{NM}U_{M}w_{N}, \  \forall \ f \in S_{k}(N,\psi).
\end{eqnarray}
The above identity follows from the definitions and the matrix identity:
\begin{equation*}
        \begin{pmatrix} 1 & 0 \\	uN & 1  \end{pmatrix}
        = -(NM)^{-1} \begin{pmatrix} 0 & -1 \\ NM & 0  \end{pmatrix} 
        \begin{pmatrix} 1 & -u \\ 0 & M  \end{pmatrix} \begin{pmatrix} 0 & -1 \\ N & 0  \end{pmatrix} .
\end{equation*}
\begin{lemma}\label{prop2.1}
	Let $f(z) = \sum_{n=1}^{\infty} a(n)e(nz) \in S_{k}(N,\psi)$  and $U_{d}$, $V_{d}$, $T_{n}$ be as above. 
	\begin{enumerate}[label={\emph{(\arabic*)}}]
	\item If $d^{2} \mid N$ and $\psi$ is a Dirichlet character mod $N/d$, then $f\vert  U_{d} \in S_{k}(N/d,\psi).$
	\item For $n \geq 1$, we have $ T_{n}(f) = \sum\limits_{ad=n} \psi(d) d ^{k-1} f \vert U_{a} V_{d}$.  Hence, 
	          $T_{p} =  U_{p}$ if  $p \mid N$ is a prime. 
    \end{enumerate}
\end{lemma}
\begin{proof} 
	 Let  $\begin{pmatrix} x & y \\	z & w	\end{pmatrix} \in \Gamma_{0}(N/d)$ and 
	 $0 \leq u,u' < d$. Then
	\begin{align*}
	       \begin{pmatrix} 1 & u \\  0 & d  \end{pmatrix}  \begin{pmatrix} x & y \\ z & w\end{pmatrix} 
	       \begin{pmatrix}  1 & u' \\ 0 & d \end{pmatrix}^{-1} 
	       = \begin{pmatrix} x+uz & \frac{y+uw-(x+uz)u'}{d} \\ dz & w-zu' \end{pmatrix}.
	\end{align*}
	 We observe that  if $d^{2} \mid N$, then $d \mid z$ and $(x,d)=1$. So $x+uz$ is a unit in 
	 $\mathbb{Z}/d    \mathbb{Z}$.
     Hence, for every $0 \leq u< d$, there exists unique $u'$ mod $d$  such that $d \mid (y+uw-(x+uz)u')$. 
This implies that for every $u$ mod $d$ there exist unique $u'$ mod $d$ such that
  \begin{align*}
  \begin{pmatrix} x_{u} & y_{u} \\  z_{u} & w_{u} \end{pmatrix} \coloneqq
  \begin{pmatrix}  x+uz & \frac{y+uw-(x+uz)u'}{d} \\ dz & w-zu' \end{pmatrix} \in \Gamma_{0}(N). 
   \end{align*}    
  Therefore 
  \begin{align*}
          (f \vert U_{d}) \vert_{k} \begin{pmatrix}x & y \\  z & w \end{pmatrix} & =
           \sum\limits_{u \ \text{mod}\  d}f \vert_{k} \begin{pmatrix} x_{u} & y_{u} \\ z_{u} & w_{u} \end{pmatrix}   
           \begin{pmatrix} 1 & u' \\ 0 & d\end{pmatrix}\\
          &=\sum\limits_{u' \ \text{mod} \ d}^{} \psi(w_{u})f \vert_{k} 
          \begin{pmatrix} 1 & u' \\ 0 & d  \end{pmatrix}\\
          &=\psi(w) \sum\limits_{u' \ \text{mod} \ d}^{} f \vert_{k} 
          \begin{pmatrix}  1 & u' \\ 0 & d \end{pmatrix} \ \ \ (\because w_{u} \equiv w \ \mathrm{mod} \ (N/d)) \\
          &=\psi(w) (f \vert U_{d}). 
  \end{align*} 
Hence (1) follows. For the  second statement we compare the Fourier expansion of both sides. 
From  the definition of $U_{a}$, $V_{d}$ it follows that
\begingroup
\allowdisplaybreaks
  \begin{alignat*}{3}
        \sum\limits_{ad= n}\psi(d) d ^{k-1} f \vert U_{a} V_{d}(z) 
        & = \sum\limits_{d \mid n}\psi(d) d ^{k-1} (f\vert U_{n/d})(dz) \ \ (\mathrm{substituting \ } a=n/d)\\
        & = \sum\limits_{d \mid n}\psi(d) d ^{k-1} \sum_{m =1}^{\infty} a(mn/d)e(mdz) \\
        & = \sum\limits_{d \mid n,  d \mid m}\psi(d) d ^{k-1} \sum_{m =1}^{\infty} a(mn/d^{2})e(mz) \\
        & = \sum_{m=1}^{\infty} \Big ( \sum\limits_{d \mid (m, n)}\psi(d) d ^{k-1} a(mn/d^{2})  \Big )e(mz) \\
        & = ( T_{n}f)(z).
   \tag*{\qedhere}
  \end{alignat*}
  \endgroup
\end{proof}
\begin{definition}
	We call an element  $f \in S_{k}(N,\psi)$ a primitive cusp form of conductor $N$ if the following conditions are satisfied:
	\begin{enumerate}[label = {\emph{(\arabic*)}}]
		\item $f$ is an eigenform, i.e., $f(z)$ is an eigenvector for the Hecke operators $T_{n}$, for all $n \in \mathbb{N},$
		\item $a(1) = 1$, where $f(z) = \sum_{n=1}^{\infty} a(n)e(nz)$,
		\item $f$ is a newform, i.e., it is orthogonal to all (old)forms lying in the 
                      images of the maps $V_{d}: S_{k}(N/d, \psi) \rightarrow S_{k}(N,\psi)$, for $d \mid N$,
                      $C_\psi \mid (N/d)$, under $\langle \> , \> \rangle_N$. 
	\end{enumerate} 
\end{definition}
If $f \in S_{k}(N,\psi)$ is a primitive cusp form,  then $ T_{q}(f) =a(q)f$ and $f \vert U_{q'} = T_{q'}(f) = a(q')f$ 
for all $q \nmid N$ and $q' \mid N$ respectively. Hence, $f$ is uniquely determined by the 
eigenvalues of the Hecke operators $T_{n}$. Further, we also have the following:
\begin{description}
	\item[Euler Product] $L(s,f) = \sum\limits_{n=1}^{\infty}a(n)e(nz)= 
	          \prod\limits_{q }^{}(1-a(q)q^{-s}+\psi(q)q^{k-1-2s}).$
	\item[Functional Equation] $ \Lambda_{N}(s;f) = i^{k} \Lambda_{N}(k-s;f \vert w_{N})$ 
	           where $\Lambda_{N}(s;f)=(2 \pi / \sqrt{N})^{-s} \Gamma(s)L(s,f).$
\end{description}
From the theory of newforms (see \cite[Theorem 4.6.15]{Miyake}) it follows that if 
$f \in S_{k}(N, \psi)$ is a primitive cusp form of conductor $C_{f}$, then 
\begin{equation}\label{rootnumber}
        f \vert w_{C_{f}} = \Lambda(f)f^{\rho},
\end{equation}  where $\Lambda(f)$ is called the root number associated to $f$. 

Let $g\in S_{k}(N,\omega)$ be a primitive cusp form of conductor $C_{g}$. If the conductor $C_{\chi}$ 
of the primitive Dirichlet character $\chi$ is coprime to $C_{g}$, 
then the twisted cusp form $g(\chi) \in S_{k}(C_{g}C_{\chi}^{2},\omega \chi^{2})$ \cite[Lemma 4.3.10 (2)]{Miyake} is primitive, and 
\begin{equation}\label{rootnumber2}
     	\Lambda(g(\chi))= \omega(C_{\chi}) \chi(C_{g})\frac{G(\chi)^{2}}{C_{\chi}}\Lambda(g), 
\end{equation}
where
\begin{align*}
    	G(\chi) = \sum\limits_{u \ \mathrm{mod} \ C_{\chi}} \chi{(u)} e^{2 \pi i u / C_{\chi}},
\end{align*}
is the Gauss sum \cite[Theorem 4.3.11]{Miyake}. 

\subsection{Rankin-Selberg convolution}

The proof of Theorem \ref{maintheorem}  makes constant use of the classical Rankin-Selberg method (see \cite{RankinSelberg}, 
\cite{Ranscalar}). For the sake of completeness, we  recall a few consequences  of the Rankin-Selberg method in this section. 
Let $f \in S_{k}(N,\psi)$, $g\in S_{l}(N,\omega)$ be primitive cusp forms. Let  $\alpha(q)$, $\alpha'(q)$, $\beta(q)$, $\beta'(q)$ 
be as  in   \eqref{alpha,beta} for all $q \nmid N$. Put $\alpha(q) = a(q)$ and $\beta(q) = b(q)$ and $\alpha'(q)= \beta'(q)=0$ 
for all $q \mid  N$. Then the $L$-function associated to $f$ and $g$ has the Euler product
\begin{equation}\label{2.8}
       \begin{split}
           L(s,f)= \sum\limits_{n=1}^{\infty} a(n)n^{-s}= \prod\limits_{q} 
           [(1-\alpha(q)q^{-s})(1-\alpha'(q)q^{-s})] ^{-1}, \\
            L(s,g)=\sum\limits_{n=1}^{\infty} b(n)n^{-s}= \prod\limits_{q} 
            [(1-\beta(q)q^{-s})(1-\beta'(q)q^{-s})] ^{-1}.
       \end{split}
\end{equation}
Before we state the result we introduce another class of Eisenstein series which are different from 
\eqref{Eisenstein series}. For every Dirichlet character $\psi$ mod $N$, we set
\begin{eqnarray}
 \label{E_{k,N}}
 E_{k,N}(z,s,\psi)= y^{s} \sideset{}{'} \sum\limits_{c,d =- \infty} ^ 
 {\infty}   \psi(d) (cNz+d)^{-k} \vert cNz+d \vert ^{-2s}.
\end{eqnarray}
Let $f \in S_{k}(N,\psi)$ and $g\in S_{l}(N,\omega)$ be primitive cusp forms as in the Introduction 
(so $l<k$) and let $\mathcal{D}(s,f,g)$ be as in \eqref{1.2}. 
Then the Rankin-Selberg method states that
\begin{enumerate}[label={(\arabic*)}]
	\item{The Rankin product $L$-function has the Euler product  
	 \begin{equation}\label{2.11}
		\begin{split}
		\mathcal{D}(s,f,g)= \prod_{q}   [  (1- & \alpha(q)\beta(q)q^{-s}) (1-\alpha(q)\beta ' (q)q^{-s}) \\
		  & \times (1-\alpha ' (q)\beta(q)q^{-s}) (1-\alpha '(q)\beta '(q)q^{-s}) ] ^{-1}.
		\end{split}
	\end{equation} }
    \item{For $s \in \mathbb{C}$ with Re$(s) > 1+\frac{k+l}{2}$, the Rankin product $L$-function 
         $\mathcal{D}(s,f,g)$ has the integral representation given by
         \begin{equation}\label{2.12}
	       2 (4 \pi )^{-s} \Gamma(s)\mathcal{D}(s,f,g)= 
	       \langle f^{\rho},gE_{k-l,N}(z,s-k+1,\psi\omega) \rangle_{N} .
       \end{equation} }
\end{enumerate}
We now state an algebraicity result for the Rankin product $L$-function which is crucial for the 
construction of the $p$-adic Rankin product $L$-function, due to Shimura.
\begin{theorem} \label{Algebraicity}  \emph{(\cite[Theorem 4]{Shi_76}, \cite[$\S$10.2, Corollary 1]{hida_93}) }
	Let $f \in S_{k}(N,\psi)$ and $g \in S_{l}(N,\omega)$ be primitive cusp forms of conductor $C_{f}$ and $C_{g}$ 
	respectively. Then for every Dirichlet character $\chi$ and for all integers $s$ with $l \leq s \leq k-1$, we have 
	\begin{align}
		\frac{\Psi(s,f,g(\chi))}{\pi ^{1-l} \langle f , f \rangle_{C_{f}} } \in \overline{\mathbb{Q}}.
	\end{align}
\end{theorem} 

\subsection{Nearly holomorphic modular forms}

In this section we recall some facts about nearly holomorphic modular forms due to Shimura (see \cite[\S 10.1]{hida_93}).
 
The Maass-Shimura  differential operator of weight $k\in \mathbb{C}$  on $C^{\infty}$-functions on $\mathcal{H}$ is the operator:
\begin{equation}
\delta_{k}= \frac{1}{2\pi i} \Big ( \frac{k}{2iy} + \frac{\partial }{\partial z} \Big ),   \ \ \mathrm{where} \ \  z= x+iy, \ 
\frac{\partial }{\partial z} =  \frac{1}{2}\Big ( \frac{\partial }{\partial x} - i\frac{\partial }{\partial y} \Big ). 
\end{equation}
For every positive integer $r$, we define $\delta_{k}^{r}:= \delta_{k+2r-2} \circ \cdots \circ \delta_{k+2}\circ \delta_{k}$ 
 and $\delta_{k}^{0}f = f$. Let $d:=\frac{1}{2 \pi i}\frac{\partial} {\partial z}$.
The Maass-Shimura differential operator satisfies the following properties:
\begin{enumerate}[label={(\arabic*)}]
	\item{$\delta_{k+s}(fg) = (\delta_{k}f)g+f(\delta_{s}g)= (\delta_{s}f)g+f(\delta_{k}g), \ \forall \     
	      s,k \in \mathbb{C}$,} 
	\item $\delta_{k}(f)= y^{-k}d(y^{k}f), \ \forall \ k \in \mathbb{C}$,
	\item $\delta_{k}^{r}= \delta_{k+2}^{r-1} \circ \delta_{k},$
	\item $\delta_{k}^{r}(f)= \sum\limits_{j=0}^{r} \binom{r}{j} \frac{\Gamma(r+k)}{\Gamma(j+k)} 
	      (-4\pi y)^{j-r}d^{j}f, \ \forall \ k \in \mathbb{C}, \ r \in \mathbb{N}.$
\end{enumerate}
\begin{definition}\label{Definition nearly holomorphic}
 Let $k$, $r$ be non-negative integers.  A function $f: \mathcal{H}\rightarrow \mathbb{C}$   
 is said to be a nearly holomorphic modular form of weight $k$ and depth less than or equal to $r$ for the 
 congruence subgroup $\Gamma$, if the following hold:
\begin{enumerate}[label={\emph{(\arabic*)}}]
 \item $f$ is smooth as a function of $x$ and $y$,
 \item $f\vert_{k}\gamma$ = $f$, \ for all $\gamma \in \Gamma$,
 \item there exist holomorphic functions $f_{0},\ldots,f_{r}$ on $\mathcal{H}$ such that 
            $ f(z) = \sum_{j=0}^{r} (4\pi y)^{-j} f_{j}(z)$,
\item f is slowly increasing, i.e.,  for every  $\alpha \in \SL_{2}(\mathbb{Z})$, there exists positive real 
numbers $A$ and $B$ such that $\lvert (f \vert_{k}\alpha)(z) \rvert \leq A (1+y^{-B})$ as $y \rightarrow \infty$.
\end{enumerate}
\end{definition}

The space of nearly holomorphic modular forms of  weight $k$ and depth less than or equal to $r$ 
for the congruence subgroup $\Gamma$ is denoted by $\mathcal{N}_{k}^{r}(\Gamma)$. It is clear 
that for $r =0$ we obtain the space of (holomorphic) modular forms $M_{k}(\Gamma)$. Let 
$\mathcal{N}_{k} (\Gamma )= \cup_{r=0}^{\infty} \mathcal{N}_{k}^{r}( \Gamma )$, then 
$\oplus_{k=0}^{\infty}\mathcal{N}_{k}(\Gamma)$ is a graded $\mathbb{C}$-algebra. Further, let 
$\mathcal{N}_{k}^{r}(N,\chi) = \lbrace f \in  \mathcal{N}_{k}^{r}(\Gamma_{1}(N)) \mid 
(f\vert_{k}\gamma)(z) = \chi(\gamma)f(z), \forall \ \gamma \in \Gamma_{0}(N)\rbrace$. 

We say a function $h \in \mathcal{N}_{k}^{r} (\Gamma)$  is rapidly decreasing if for every 
$B \in \mathbb{R}$ and $\alpha \in \SL_{2}(\mathbb{Z})$, there exists a positive constant $A$ such 
that  $\lvert (h \vert_{k}\alpha)(z) \rvert \leq A (1+y^{B})$ as $y \rightarrow \infty$. We denote the 
subspace of  rapidly decreasing functions in $\mathcal{N}_{k}^{r}(\Gamma)$, $\mathcal{N}_{k}^{r}(N,\chi)$ 
and $\mathcal{N}_{k}(\Gamma)$ by $\mathcal{NS}_{k}^{r}(\Gamma)$, $\mathcal{NS}_{k}^{r}(N,\chi)$ and
 $\mathcal{NS}_{k}(\Gamma)$ respectively (cf. Lemma~\ref{Lemmabounds}).

\begin{lemma}\label{Mass-Shimura Operator}
   If $h: \mathcal{H} \rightarrow \mathbb{C}$ is a $C^\infty$-function, then 
   $(\delta_{k}^{r}h)\vert_{k+2r} \gamma  = \delta_{k}^{r}(h\vert_{k}\gamma) $,  
   for all $\gamma \in \gl_{2}^{+}(\mathbb{R}).$
\end{lemma}
\begin{proof}
	   Observe that by induction on $r$, it is enough to prove the lemma for $r=1$ and for all 
	   $k \in \mathbb{C}$. So, it is enough to show
	   $$
	          (\delta_{k}h)\vert_{k+2} \gamma = \  \delta_{k}(h\vert_{k}\gamma).
	   $$
	   For $\gamma= \begin{pmatrix}  a & b \\c & d \end{pmatrix} \in \gl_{2}^{+}(\mathbb{R})$,
	   the left hand side is given by 
	   \begin{align*}
	       ((\delta_{k} h)\vert_{k+2}\gamma )(z) \  
	         & = \ \frac{1}{2\pi i} \Big ( \Big ( \frac{kh}{2i\mathrm{Im}(z)} 
	                + \frac{\partial h}{\partial z} \Big )\Big \vert_{k+2}  \gamma \Big )(z) \\
	        & = \ \frac{1}{2\pi i}  \Big ((cz+d)^{-k-2} \frac{k}{2i\mathrm{Im}(\gamma z)} h(\gamma z) 
	          + (cz+d)^{-k-2}\frac{\partial h}{\partial z} (\gamma z)\Big ) \\
	       & = \ \frac{1}{2\pi i}  \Big ((cz+d)^{-k-2}  \vert cz+d \vert ^{2}\frac{k}{2iy} h(\gamma z) 
	         + (cz+d)^{-k-2}\frac{\partial h}{\partial  z}(\gamma z) \Big ).
	   \end{align*}
	   The right hand side is given by 
	   \begin{align*}
	   \delta_{k}(h\vert_{k}\gamma) (z) \ & = \delta_{k}((cz+d)^{-k}h(\gamma z)) \\
	   &= \frac{1}{2\pi i} \Big ( \frac{k}{2iy} + \frac{\partial }{\partial z} \Big ) ((cz+d)^{-k}h(\gamma z)) \\
	   & =\frac{1}{2\pi i} \Big ( \frac{k}{2iy} (cz+d)^{-k}h(\gamma z) -ck(cz+d)^{-k-1}h(\gamma z)
	        + (cz+d)^{-k-2}\frac{\partial h}{\partial z}(\gamma z)\Big ) \\	
	   & = \ \frac{1}{2\pi i}  \Big ( \frac{k}{2iy} (cz+d)^{-k-1}  h(\gamma z) (cx+ciy+d-2ciy)
	        + (cz+d)^{-k-2}\frac{\partial h}{\partial  z}(\gamma z) \Big ) \\
	   & = \ \frac{1}{2\pi i}  \Big ( \frac{k}{2iy} (cz+d)^{-k-2}  \vert cz+d \vert ^{2} h(\gamma z) 
	        + (cz+d)^{-k-2}\frac{\partial h}{\partial  z}(\gamma z) \Big ), 
	   \end{align*}
	   which proves that $(\delta_{k}h)\vert_{k+2} \gamma = \  \delta_{k}(h\vert_{k}\gamma)$ and completes the proof.
\end{proof}
Let $h : \mathcal{H} \rightarrow \mathbb{C}$ be a holomorphic function such that $h(z)=\sum_{n=0}^{\infty} a(n) e(nz/N)$.  
Then $e(-z/N)(h(z) -a_{0})$ is holomorphic on $\mathcal{H} \cup \lbrace \infty \rbrace$.  
Thus, there exists a positive real number $C$ such that
\begin{align}\label{slowly increasing}
      \lvert h(z) \rvert \leq \lvert h(\infty) \rvert + C e^{-2\pi  y/N} \ \mathrm{as} \ y \rightarrow \infty.
\end{align}
\begin{proposition} \label{prop2.6}
For $k$, $r \in \mathbb{N}$, the operator $\delta_k^r$ induces a linear map of $\mathbb{C}$-vector spaces
 $\delta_{k}^{r}:M_{k}(\Gamma) \rightarrow \mathcal{N}_{k+2r}^{r}(\Gamma)$.
\end{proposition}
\begin{proof}
Clearly $\delta_{k}^{r}$ is $\mathbb{C}$-linear. So it is enough to show $\delta_{k}^{r}(f) \in \mathcal{N}_{k+2r}^{r}(\Gamma), \ \forall  \ f \in M_{k}(\Gamma)$. 
Let $f\in M_{k}(\Gamma)$.  
Recall that
$$
     \delta_{k}^{r}(f)= \sum\limits_{j=0}^{r} \binom{r}{j} \frac{\Gamma(r+k)}{\Gamma(j+k)} (-4\pi y)^{j-r}d^{j}f.
 $$
  Clearly $d^{j}f$ is holomorphic and $y^{j-r}$ is smooth. Hence, $\delta_{k}^{r}(f)$ satisfies (1) and (3) of Definition \ref{Definition nearly holomorphic}.  From  Lemma \ref{Mass-Shimura Operator},  it follows that 
  \begin{align}\label{Propn2.8 eq1}
  (\delta_{k}^{r}f)\vert_{k+2r} \gamma =   \delta_{k}^{r}(f\vert_{k}\gamma) =\delta_{k}^{r}(f), \mathrm{ \ for \ all } \ \gamma \in \Gamma,
  \end{align}
hence (2) also holds.  
It remains to check that  $\delta_{k}^{r}f$ is slowly increasing. If $\alpha \in \SL_{2}(\mathbb{Z})$, then $f \vert_k \alpha$ is also $C^\infty$, so  
  $$
        (\delta_{k}^{r}f)\vert_{k+2r} \alpha= \delta_{k}^{r}(f\vert_{k}\alpha)= \sum\limits_{j=0}^{r} \binom{r}{j} \frac{\Gamma(r+k)}{\Gamma(j+k)} (-4\pi y)^{j-r}d^{j}(f\vert_{k}\alpha).
  $$
Note that the $(-4 \pi y)^{j-r}$ are bounded as $y \rightarrow \infty$ and the $d^{j}(f\vert_{k}\alpha)$ are holomorphic.
It follows from \eqref{slowly increasing} that, for every $0\leq j \leq r$, there exists  
positive numbers $A_{j},B_{j}$ such that  $\lvert (-4\pi y)^{j-r}d^{j}(f\vert_{k}\alpha)\rvert 
\leq A_{j}(1+e^{-2 \pi y / B_{j}}) $  as $y \rightarrow \infty$. 
Since $e^{-y}$ decays faster than  $y^{-n}$ for any $n \geq 0$ as $y \rightarrow \infty$, we have
$\lvert (\delta_{k}^{r}f)\vert_{k+2r} \alpha \rvert \leq A_{\alpha} (1+y^{-B_{\alpha}})$ as $y \rightarrow \infty$ 
for some positive numbers $A_{\alpha}$, $B_{\alpha}$. 
\end{proof}

Now we will show that $E_{k}(z,s;\chi,\chi_{0})$ is a nearly holomorphic modular form if $\chi$ is a 
Dirichlet character modulo $N$ and $s \leq 0$ is an integer such that $k+2s > 2$. To prove this, we 
need to consider the action of the Maass-Shimura operator on Eisenstein series. Observe that for 
$k$, $r$ positive integers and $s \leq 0$ an integer such that $k+2s > 2$, we have 
\begin{align}\label{2.13} 
        \delta_{k}^{r}(y^{s})  
        & =  \sum\limits_{j=0}^{r} \binom{r}{j} \frac{\Gamma(k+r)}{\Gamma(k+j)} 
	     (-4\pi y)^{j-r}d^{j}y^{s}  \nonumber \\
	& =  \sum\limits_{j=0}^{r} \binom{r}{j} \frac{\Gamma(k+r)}{\Gamma(k+j)} 
	     (-4\pi y)^{j-r} \left( \frac{-1}{ 4 \pi} \right)^{j} 
	     \frac{\Gamma(s+1)}{\Gamma(s-j+1)} y^{s-j} \nonumber \\
	& =  (-4 \pi)^{-r} y^{s-r} r! \sum\limits_{j=0}^{r} \binom{k+r-1}{r-j} \binom{s}{j} 
	      \nonumber \\
	& =  (-4 \pi)^{-r} \frac{\Gamma(s+k+r)}{\Gamma(s+k)} y^{s-r}, 
\end{align}
where the last equality follows by comparing the coefficient of $X^{r}$ in $(1+X)^{s} (1+X)^{k+r-1}$ and 
$(1+X)^{s+k+r-1}$. 
For $(c,d) \in \mathbb{Z}^{2} \setminus \lbrace (0,0) \rbrace$, let  
$\gamma = \begin{psmallmatrix} d & -c \\ c & d \end{psmallmatrix}$. 
Since $y^{s} \vert_{k} \gamma = (cz+d)^{-k} \lvert cz+d \rvert^{-2s} y^{s}$, we have
\begin{eqnarray*}
       \delta_{k}^{r}((cz+d)^{-k}\vert cz+d \vert^{-2s} y^{s}) 
       & = &  \delta_{k}^{r} (y^{s} \vert_{k} \gamma) 
        ~ = ~ \delta_{k}^{r} (y^{s}) \vert_{k+2r} \gamma \ \ \  \ \left(\mathrm{By \ Lemma} \ 
              \ref{Mass-Shimura Operator} \right) \\
       & \stackrel{\eqref{2.13}}{=} & (-4 \pi)^{-r} \frac{\Gamma(s+k+r)}{\Gamma(s+k)} 
          (y^{s-r}) \vert_{k+2r} \gamma \\
       & = & (-4 \pi)^{-r} \frac{\Gamma(s+k+r)}{\Gamma(s+k)}
               (cz+d)^{-k-2r}\vert cz+d \vert^{-2(s-r)}  y^{s-r}. 
\end{eqnarray*}
Let $\chi$ be  a Dirichlet character modulo $N$. Multiplying both sides of the equation above by $\chi(c)$ 
and then taking the sum over all $(c,d) \in \mathbb{Z}^{2} \smallsetminus \lbrace (0,0)\rbrace$ (ignoring 
convergence issues) we get 
 \begin{equation}\label{2.18}
 	(-4\pi)^{r}\frac{\Gamma(s+k)} {\Gamma(s+k+r)} \delta_{k}^{r}(E_{k}(z,s;\chi,\chi_{0}))
 	 = E_{k+2r}(z,s-r;\chi,\chi_{0}). 
 \end{equation}
 From \cite[Chapter 7]{Miyake} we know that  if $k \geq 3$, $E_{k}(z,0; \chi,\chi_{0}) = \sum_{c,d}^{'}  \chi(c) (cz+d)^{-k}$ 
is a usual holomorphic modular form in $M_{k}(N,\chi)$. 
 It follows from \eqref{2.18} and Proposition \ref*{prop2.6} that for all integers $r \geq 0$ and
 $k \geq 3$:
\begin{align*}
        E_{k+2r}(z,-r;\chi,\chi_{0}) \in \mathcal{N}_{k+2r}^{r}(N, \chi) .
\end{align*}
A similar argument for the Eisenstein series $E_{k,N}(z,0,\overline{\chi})  \in M_{k}(N,\chi)$,  $k \geq 3$ (cf. \eqref{E_{k,N}}), gives: 
%
\begin{proposition}\label{prop2.9}
	Let $k$, $r$ be integers such that $0 \leq r < k/2-1$. If $\chi$ is a Dirichlet character mod $N$, 
        then $E_{k}(z,-r; \chi, \chi_{0})$, $E_{k,N}(z,-r,\overline{\chi})  \in \mathcal{N}_{k}^{r}(N,\chi)$.
\end{proposition}
 \begin{theorem} \emph{\cite[$\S$10.1, Theorem 1]{hida_93}} \label{thm2.8}
 	Suppose that $r \geq 0$ and $k \geq 1$. If $f \in \mathcal{N}_{k+2r}^{r}(N,\chi)$, then  
 	\begin{align}\label{2.20}
 	     f = \sum_{j=0}^{r} \delta_{k+2r-2j}^{j}h_{j}, \  \mathrm{where} \ h_{j}  
 	     \in M_{k+2r-2j}(N,\chi).
 	 \end{align}  
 	More precisely, 
       \begin{align*}
            \mathcal{N}_{k+2r}^{r}(N,\chi) \cong \oplus_{j=0}^{r} M_{k+2r-2j}(N,\chi),  \\	
            \mathcal{N}S_{k+2r}^{r}(N,\chi) \cong \oplus_{j=0}^{r} S_{k+2r-2j}(N,\chi), 
       \end{align*}
       and the isomorphism is obtained via $(f) \mapsto (h_{j})$, where $h_{j}$ are as in \eqref{2.20}. 
       Moreover these isomorphisms are equivariant under the $\vert_k$ action of $\mathrm{GL}_{2}^{+}
       (\mathbb{R})$.
 \end{theorem}
\noindent The projection $ f  \mapsto h_{0} $ induces  a map
\begin{align}
   \label{HolN}
	{\mathcal{H}ol}: \mathcal{N}_{k+2r}^{r}(N,\chi) \rightarrow M_{k+2r}(N,\chi),
\end{align}
which is called the holomorphic projection.
%
\begin{lemma}\label{lemma2.11}
	Let $f \in S_k(N,\psi)$ and let $g : \mathcal{H} \rightarrow \mathbb{C}$ be a smooth function which is 
        slowly increasing 
	such that $g\vert_{k}\gamma =\psi(\gamma) g $ for every $\gamma \in \Gamma_{0}(N)$. Then 
	$\langle f , g \rangle_{N} := \int_{\mathcal{H}/\Gamma_{0}(N)} \overline{f(z)}g(z) y^{k-2} dx  dy$ converges. 
\end{lemma}

\begin{proof}
  This follows from  Lemma~\ref{Lemmabounds} (1) below and \cite[$\S$9.3, (6)]{hida_93}.
\end{proof}

\begin{lemma}\label{lemma2.10}
       Suppose $f \in S_{k}(N, \chi)$ and $g \in \mathcal{N}_{k}^{r}(N,\chi)$. If $r < k/2$, then 
       $\langle f,g \rangle_N = \langle f, \mathcal{H}ol(g) \rangle_N$. Further, if $g \in 
       \mathcal{N}S_{k}^{r}(N,\chi)$, then $\mathcal{H}ol(g)$ is the unique cusp form with the property 
       $\langle  f , g \rangle_{N} = \langle f , \mathcal{H}ol(g) \rangle_{N},\ \forall \ f \in 
       S_{k}(N, \chi)$.
\end{lemma}

\begin{proof}
      The first part  follows from \cite[$\S$10.1, Corolllary 1]{hida_93}. By Theorem \ref{thm2.8}, we have 
      $\mathcal{H}ol(g) \in S_{k}(N,\chi)$. Now the first part shows that $\mathcal{H}ol(g)$ satisfies the 
      required property. From  Lemma \ref{lemma2.11}, we have $f \mapsto \langle f,g \rangle_{N}$ 
      defines an anti-linear functional on $S_{k}(N,\chi)$.  The uniqueness statement follows from the fact that  
      Petersson inner product induces a non-degenerate pairing $ \langle  \>\>  , \>\> \rangle_N : S_{k}(N, \chi ) 
      \times S_{k}(N,\chi) \rightarrow \mathbb{C}$.
\end{proof}
\begin{lemma} \label{hol} \textbf{\emph{(}Holomorphic Projection lemma\emph{)}} \
              \emph{\cite[Appendix C]{Zagier}}
	Let $\varPhi:\mathcal{H} \rightarrow \mathbb{C}$ be a smooth function satisfying:
	\begin{enumerate}[label={\emph{(\arabic*)}}]
		\item  $\varPhi(\gamma(z)) = (cz+d)^{k} \varPhi(z), \  \forall \ \gamma = 
		       \begin{pmatrix} a & b \\ c & d \end{pmatrix} \in \Gamma$  and $ z \in \mathcal{H}$,
		\item $\varPhi(z) = c_{0}+O(y^{-\epsilon})$  as  $y =\mathrm{Im(z)} \rightarrow \infty$,
	\end{enumerate}
       for some integer $k > 2$ and numbers $c_{0} \in \mathbb{C}$ and $\epsilon > 0$. If 
       $\varPhi(z)  = \sum_{n=0}^{\infty} c_{n}(y)e(n x)$, then the function 
       $\phi(z):= \sum_{n=0}^{\infty}c_{n}e(nz)$ with  $$c_{n}= \frac{(4 \pi n)^{k-1}}{(k-2)!} 
       \int_{0}^{\infty} c_{n}(y)e^{-2 \pi n y}y^{k-2} dy $$ for $n >0$ belongs to $M_{k}(\Gamma)$  and 
       satisfies $\langle f , \phi\rangle _{\Gamma}= \langle f, \varPhi \rangle_{\Gamma},\ \forall  
       \ f \in S_{k}(\Gamma).$
\end{lemma}
Any rapidly decreasing function $\Phi: \mathcal{H} \rightarrow \mathbb{C}$ which satisfies  hypothesis (1)
of Lemma~\ref{hol}, automatically satisfies hypothesis (2) with $c_0 = 0$. 
For such $\Phi$, we set  $$\mathcal{H}ol(\Phi) : = \phi,$$ 
where $\phi$ is as defined in Lemma~\ref{hol}. It is easy to see that $\phi$ is a cusp form. Recall that the 
elements of $\mathcal{N}S_{k}^{r}(N,\chi)$ are rapidly decreasing.
The definition of $\mathcal{H}ol$ given just above in fact extends the 
definition of the holomorphic projection $\mathcal{H}ol$ given in \eqref{HolN}, by the uniqueness part of 
Lemma~\ref{lemma2.10}.


We now state a result 
which will enable one to apply the lemma above.
%
\begin{lemma}\label{Lemmabounds}
       Let $k$, $N$ be a positive integers and $\chi$ a Dirichlet character mod $N$. Then
       \begin{enumerate}[label={\emph{(\arabic*)}}]
	      \item If $h \in S_{k}(N,\chi)$, then  $\lvert (h \vert_{k} \gamma)(z)\rvert  =  O(y^{-B})$, 
	            for all positive real numbers $B$ and all $\gamma \in \SL_{2}(\mathbb{Z})$, 
	            as $y \rightarrow \infty$. In particular $h$ is rapidly decreasing.
	      \item For any compact set $T \subset \mathbb{R}$ and $\gamma \in \SL_{2}(\mathbb{Z})$, there 
	            exists positive real numbers $A$ and $B$ such that if $\chi \neq \chi_{0}$ 
		    $$
			\lvert E_{k}(z,s;\chi,\chi_{0})\vert_{k}\gamma \rvert \leq A(1+y^{-B}), 
			\mathrm{ \ as}\  y \rightarrow \infty \mathrm{ \ as \ long \ as \  Re}(z) \in T.
		    $$
	\end{enumerate}   
\end{lemma}
\begin{proof}
       Observe that if $h \in S_{k}(N,\chi)$, then $h$ vanishes at the cusps. Now, the first part of the 
       lemma follows from \eqref{slowly increasing}. For the second part see \cite[$\S$9.3, Lemma 3]{hida_93}.
\end{proof}
 It follows from Lemma~\ref{Lemmabounds} that if $h$ is a (holomorphic) cusp form of weight $2 \leq l < k $
(in our application below $h$ will be the slash of a twist of $g$ from the Introduction), then 
$h(z) E_{k-l}(z,s;\chi,\chi_{0})$ has weight $k > 2$ and satisfies the hypotheses of Lemma \ref{hol}, with 
$c_{0}=0$. So $\mathcal{H}ol(h(z)E_{k-l}(z,s;\chi,\chi_{0}))$ is defined, and we can calculate its Fourier expansion 
using Lemma~\ref{hol} if we know the Fourier expansion of $h(z)E_{k-l}(z,s;\chi,\chi_{0})$.
\section{Distributions and Measures}
In this section, we define distributions and measures following \cite{Panrankin}. Most of the material 
covered in this section can also be found in \cite{Wash97}, \cite{MSD}. Finally, we state the abstract 
Kummer congruences which is the key tool used in the construction of the $p$-adic $L$-function.   
\subsection{Distributions}
 Let $Y$ be a  
compact, Hausdorff and totally disconnected topological space. Then $Y$ is a projective limit of finite discrete 
spaces $Y_i$, 
\begin{equation}
       Y = \varprojlim \ Y_{i}, 
\end{equation}
with respect to transition maps $\pi_{ij}: Y_{i} \rightarrow Y_{j}$, for $i \geq j$, $i$, $j$ in some directed set $I$. 
We assume that the $\pi_{ij}$ are surjections, so the canonical maps $\pi_{i} : Y \rightarrow Y_{i}$ are
projections. Let $R$ be a commutative ring and let Step($Y,R$) be the set of $R$-valued 
locally constant functions on $Y$. 
\begin{definition}
       A distribution on $Y$ with values in an $R$-module $\mathcal{A}$  is a homomorphism of 
       $R$-modules 
      $$
	  \mu: \mathrm{Step}(Y,R) \rightarrow \mathcal{A}.
      $$
\end{definition}
\noindent We use the notation 
\begin{align*}
       \mu(\varphi)= \int\limits_{Y}^{} \varphi \ d \mu = \int\limits_{Y} \varphi(y) \ d\mu(y),
\end{align*}
for $\varphi \in \mathrm{Step}(Y,R)$. Any distribution $\mu$ can be given by a system of functions 
$\lbrace \mu^{(i)}:Y_{i} \rightarrow \mathcal{A} \rbrace$, satisfying the following finite additivity
condition:
\begin{equation}\label{finadd}
       \mu^{(j)}(y)= \sum\limits_{x \in \pi_{ij}^{-1}(y)} \mu^{(i)}(x), \  \forall \ y \in Y_{j}, \  
        x \in Y_{i}, \ i \geq j.
\end{equation}
Indeed, given such a system of functions $\lbrace \mu^{(i)}: Y_{i} \rightarrow 
\mathcal{A} \mid i \in I \rbrace$, if $\delta_{i,x}$ is the characteristic function of the inverse image 
$\pi_{i}^{-1}(x) \subset Y$, for $x \in Y_{i}$, define
$$
	\mu(\delta_{i,x}) = \mu^{(i)}(x)
$$
and extend the definition of $\mu$ to all of $\mathrm{Step}(Y,R)$ by linearity. 
Conversely, given  a distribution $\mu$, in order to construct such a system, set
$\mu^{(i)}(x) =  \mu(\delta_{i,x})  \in  \mathcal{A}, \  \forall \ x \in Y_{i}$. 

It can be checked that a system of functions $ \lbrace \mu^{(i)}:Y_{i} \rightarrow\mathcal{A} \rbrace$
 satisfies \eqref{finadd} if and only if for all $j \in I$ and all $\varphi_{j}:Y_{j} \rightarrow R $, 
\begin{align}\label{finadd1}
        \mathrm{the \ sum} \ 
          \sum_{x \in Y_{i}}^{} \varphi_{i}(x) \mu^{(i)}(x) \  
          \mathrm{does\ not \ depend \ on \ } i,\ \forall \ i \geq j, 
\end{align}
where $\varphi_{i} := \varphi_{j}\circ \pi_{ij} : Y_i \rightarrow R$.
If $\mu$ is the corresponding distribution and $\varphi = \varphi_{j} \circ \pi_{j} \in$ Step$(Y,R)$, 
then $\mu(\varphi)$ is just the sum above. 
If $Y = G = \varprojlim \ G_{i}$ is a profinite abelian group and $R$ is an integral domain containing all 
roots of unity of order dividing the cardinality of $G$ (perhaps a transfinite cardinal, in which case 
$R$ contains all roots of unity), then one needs to 
verify \eqref{finadd1} only for all characters of finite order $\chi: G \rightarrow R^{\times}$,  since the 
orthogonality relations imply that their linear span over $R \otimes \mathbb{Q}$ coincides with 
Step($Y,R \otimes \mathbb{Q}$) (see \cite{MSD}).  
\begin{example}\label{example}
	Let  $p$ be an odd prime. Then $\mathbb{Z}_{p}^{\times} = \varprojlim \left( \mathbb{Z}/p^{n} 
	\mathbb{Z} \right) ^{\times}$. We consider 
	\begin{align*}
		X_{p} = X(\mathbb{Z}_{p}^{\times}) = \mathrm{Hom}_{\mathrm{cont}} (\mathbb{Z}_{p}^{\times}, 
		\mathbb{C}_{p}^{\times}), \ 
		\mathcal{B} = \lbrace \chi \in X(\mathbb{Z}_{p}^{\times}) \mid  \chi \ 
		\mathrm{has \ finite \ order}\rbrace.
	\end{align*}
    We claim that $\mathcal{B}$ is a basis for Step$(\mathbb{Z}_{p}^{\times} , \mathbb{C}_{p})$ as a 
    $\mathbb{C}_{p}$-vector space. For 	every $x \in ( \mathbb{Z}/p^{n} \mathbb{Z} ) ^{\times}$, let 
    $\delta_{n,x}$ be the characteristic function  of the basic open set 
	$\lbrace a \in \mathbb{Z}_{p}^{\times} \mid a \equiv x \ \mathrm{mod} \ p^{n} \rbrace$. Then, by the 
	orthogonality relations, we have 
	\begin{align}
	       \delta_{n,x} = \frac{1}{\varphi (p^{n})} \sum\limits_{\chi \in X((\mathbb{Z}/ p^{n} 
	       \mathbb{Z})^{\times})} \bar{\chi} (x)\ \chi. 
	\end{align}
        Since 
	every locally constant function $\mathbb{Z}_{p}^{\times} \rightarrow \mathbb{C}_{p}$ is a $\mathbb{C}_{p}$-linear 
	combination of characteristic functions, we see that $\mathcal{B}$ spans 
	Step$(\mathbb{Z}_{p}^{\times}, \mathbb{C}_{p})$. For linear independence, let $\chi_{1}, \ldots , \chi_{n} \in \mathcal{B}$ 
	and suppose $\sum_{i=1}^{n } a_{i} \chi_{i} = 0$, with $a_{i} \in \mathbb{C}_{p}$. By choosing  $m$ 
	sufficiently large we may assume $\chi_{i} \in X((\mathbb{Z}/ p^{m} \mathbb{Z})^ \times)$, for all 
	$i$.  By linear independence of characters, we have $a_{i} = 0$, for all $i$. 
\end{example}
\subsection{Measures}
Let $R$ be a topological ring with  topology induced by a norm.  Let $\mathcal{C}(Y,R)$  denote the 
$R$-module of continuous  $R$-valued functions on $Y$ and equip $\mathcal{C}(Y,R)$ with the 
corresponding sup norm topology. In this article we will take $R= \mathbb{C}$ (or) $\mathbb{C}_{p}$
(or) $\mathcal{O}_{p}:= \lbrace x \in \mathbb{C}_{p} \mid \lvert x \rvert_{p} \leq 1 \rbrace $. 
\begin{definition}
	A measure on $Y$ with values in a topological $R$-module $\mathcal{A} $ is a continuous 
	homomorphism of $R$-modules $\mu: \mathcal{C}(Y,R) \rightarrow \mathcal{A}$.
\end{definition}
The restriction of a measure $\mu$ to the $R$-submodule Step($Y,R$) $\subset$ $\mathcal{C}(Y,R)$ is a 
distribution, which we denote by the same symbol. Since $Y$ is compact, we have Step($Y,R$) is dense
in $\mathcal{C}(Y,R)$. So every measure is uniquely determined by its values on  Step($Y,R$).
We take for $R$ a closed subring of $\mathbb{C}_{p}$, and 
let $\mathcal{A}$ be a complete $R$-module with topology induced by a norm 
$\mid \cdot \mid _{\mathcal{A}}$ on $\mathcal{A}$. We further assume that $\mid \cdot \mid _{\mathcal{A}}$  
is compatible with $\mid \cdot \mid _{p}$, i.e., $\lvert ra \rvert_{\mathcal{A}} = \lvert r \rvert _{p} 
\lvert a \rvert _{\mathcal{A}} $ for all $r \in R$ and $a \in \mathcal{A}$. 
Then the condition that a distribution  
$\lbrace \mu^{(i)}:Y_{i} \rightarrow \mathcal{A} \rbrace$ gives rise to an $\mathcal{A}$-valued measure
on $Y$ is equivalent to the condition that the $\mu^{(i)}$ are bounded, i.e., there is a uniform constant
$B>0$ such that for all $i \in I$ and all $x \in Y_{i}$, we have 
$\lvert \mu^{(i)}(x)\rvert_{\mathcal{A}} < B$. The proof of this fact is easy using the non-archimedean property 
and completeness of the norm $\mid \cdot \mid_\mathcal{A}$ (see \cite[Proposition 12.1]{Wash97}). In particular, if
$\mathcal{A}= R= \mathcal{O}_{p}= \lbrace x \in \mathbb{C}_{p} \mid \lvert x \rvert_{p} \leq 1 \rbrace$ is 
the ring of integers of $\mathbb{C}_{p}$, then distributions are the same as 
measures. The most important tool in the construction of the $p$-adic $L$-function 
is the following criterion for the existence of a measure with prescribed properties. 
\begin{theorem} \textbf{\emph{(}The abstract Kummer congruences\emph{)}}
                  \emph{(\cite[Proposition 4.0.6]{Katzp-adic}, \cite{Panchiskhinbook})}
	Let $\lbrace f_{i} \rbrace$ be a system of continuous $\mathcal{O}_{p}$-valued functions on $Y$ such that
	 the $\mathbb{C}_{p}$-linear  span of $\lbrace f_{i} \rbrace$ is dense  in $\mathcal{C}(Y,\mathbb{C}_{p})$. Let 
	 $\lbrace a_{i} \rbrace$ be any system of elements with $a_{i} \in \mathcal{O}_{p}$. Then the existence of an
	  $\mathcal{O}_{p}$-valued measure $\mu$ on $Y$ \emph{(}i.e., $\mu (\mathcal{C}(Y,\mathcal{O}_{p})) 
	  \subset \mathcal{O}_p$\emph{)}  with the property that
	 \begin{equation*}
		\int_{Y} f_{i} \ d\mu =a_{i}
	\end{equation*}
	is equivalent to the following: for an arbitrary choice of elements $b_{i} \in \mathbb{C}_{p}$, 
	almost all of which vanish, and any $n \geq 0$, we have the following implication of congruences:
	\begin{equation}\label{3.4}
		\sum\limits_{i}b_{i}f_{i}(y) \in p^{n} \mathcal{O}_{p},  \ \forall \ y \in Y \ \Longrightarrow \  
		\sum\limits_{i} b_{i}a_{i} \in p^{n}\mathcal{O}_{p}.
	\end{equation}
\end{theorem}	
\begin{proof}
      The necessity is obvious. Indeed if $\sum\limits_{i}b_{i}f_{i}(y) \in p^{n} \mathcal{O}_{p}$, then
       \begin{align*}
	      \sum\limits_{i} b_{i}a_{i} 
	       & = \sum\limits_{i} \int\limits_{Y} b_{i}f_{i} \ d \mu   \\
	       & = p^{n} \int\limits_{Y} \left( p^{-n}\sum\limits_{i} b_{i}f_{i} \right)d \mu 
	           \ \in p^{n} \mathcal{O}_{p}.
       \end{align*}
       In order to prove the sufficiency we need to construct a measure $\mu$ from the numbers $a_{i}$. 
       For a function $f \in \mathcal{C}(Y,\mathcal{O}_{p})$ and a positive integer $n$, there exists 
       $b_{i} \in \mathbb{C}_{p}$ such that $b_{i}=0$ for almost all $i$, and
      \begin{align*}
             f- \sum\limits_{i}b_{i}f_{i} \in \mathcal{C}(Y,p^{n} \mathcal{O}_{p})
       \end{align*}
      by the density of the $\mathbb{C}_{p}$-span of  the $\lbrace f_{i} \rbrace$ in 
      $\mathcal{C}(Y,\mathbb{C}_{p})$. Now, we claim that the value $\sum_i b_{i}a_{i}$ belongs to 
      $\mathcal{O}_{p}$ and is well defined modulo $p^{n}$, i.e., it doesn't depend on the choice of 
      $b_{i}$. Since $f \in \mathcal{C}(Y,\mathcal{O}_{p})$, clearly 
      $\sum_{i}b_{i}f_{i} \in \mathcal{C}(Y,\mathcal{O}_{p})$. 
      Therefore, by \eqref{3.4}, we have $\sum_{i}b_{i}a_{i} \in \mathcal{O}_{p}$.  Let 
      $c_{i} \in \mathbb{C}_{p}$ be another set of numbers with $c_{i} \neq 0$ only for finitely many 
      $i$ such that $f- \sum_{i}^{} c_{i}f_{i} \in \mathcal{C}(Y, p^n \mathcal{O}_p)$. Then 
     \begin{align*}
	   \sum\limits_{i}(c_{i}-b_i)f_{i} 
	    = (f- \sum\limits_{i}b_{i}f_{i})- (f- \sum\limits_{i}c_{i}f_{i} ) \in  
	   \mathcal{C}(Y,p^{n} \mathcal{O}_{p}). 
      \end{align*}
     By  \eqref{3.4}, we have 
     $\sum_i c_{i}a_{i}  \equiv \sum_i b_{i}a_{i} \ \mathrm{mod} \ p^{n} \mathcal{O}_{p}$. 
     Therefore, $\sum_{i}^{} b_{i}a_{i}$ is well defined modulo $p^{n}$. We denote this value by 
     $ \int_{Y} f d\mu \ \text{mod}\ p^{n} $. Further, the above argument shows 
     $ (\int_{Y} f d\mu \ \text{ mod}\ p^{n+1})  \equiv (\int_{Y} f d\mu \ \text{mod}\ p^{n})$ mod 
     $p^{n}\mathcal{O}_{p}$. So, we may define $\mu$ on  $\mathcal{C}(Y,\mathcal{O}_{p})$ via 
     $$
         \int_{Y} f \ d\mu = 
         \left\{ \left( \int_{Y} f \ d \mu \ \mathrm{mod}\ p^{n}\right) \right\}_{n \geq 1}    
         \in \varprojlim  \> \mathcal{O}_{p}/p^{n}\mathcal{O}_{p}  = \mathcal{O}_{p}.
     $$
     Since every element of $\mathcal{C}(Y,\mathbb{C}_{p})$ is bounded, by rescaling, the above definition 
     of $\mu$ extends to all  of $\mathcal{C}(Y,\mathbb{C}_{p})$. A check shows 
     $\mu :  \mathcal{C}(Y,\mathbb{C}_{p}) \rightarrow \mathbb{C}_{p}$ is continuous, so $\mu$ is an 
     $\mathcal{O}_{p}$-valued measure. Clearly $\int_{Y} f_{i} \  d\mu=a_{i}$.
\end{proof}
Recall that $X_p = \mathrm{Hom}_{\mathrm{cont}} (\mathbb{Z}_{p}^{\times},  \mathbb{C}_{p}^\times)$ has 
an analytic structure described in the Introduction. If 
$\mu :  \mathcal{C}( \mathbb{Z}_p^\times,\mathbb{C}_{p}) \rightarrow \mathbb{C}_{p}$ is a measure, 
the non-archimedean Mellin transform of $\mu$,  defined by
\begin{align}\label{3.5}
       L_{\mu}(\chi) = \  \mu(\chi) = \ \int_{\mathbb{Z}_{p}^{\times}} \chi \ d\mu, 
       \ \forall \  \chi \in X_{p},
\end{align}
gives a bounded $\mathbb{C}_{p}$-analytic function $L_{\mu}:X_{p}\longrightarrow \mathbb{C}_{p}$
(see \cite[$\S$ 7.4]{MSD}, \cite[Theorem 8.7]{Manin}). 
Here `analytic' means that the integral \eqref{3.5} depends analytically on the parameter $\chi \in X_{p}$. 
The converse is also true: any bounded $\mathbb{C}_p$-analytic function on $X_{p}$ is the Mellin transform 
of some measure $\mu$. These measures with the convolution operation form an algebra, which essentially 
coincides with the Iwasawa algebra (see  \cite[$\S$(1.4.3), $\S$(1.5.2)]{Panchiskhinbook}). 
%
\section{ Construction of Complex-valued Distributions  } 
From now on, let $f$ and $g$ be the primitive cusp forms as in the Introduction. 
In this section we define two complex-valued distributions associated to $f$ and $g$ and compare them. 
 
Let $p$ be a prime as in the Introduction. The $p$-stabilization of $f$ is defined by
\begin{equation}\label{4.1}
	f_{0}(z)= f(z)-\alpha'(p) f(pz) = f(z)-\alpha'(p) (f\vert V_{p})(z)   ,
\end{equation}
where as before $f(z) = \sum_{n=1}^{\infty}a(n,f)e(nz) \in S_{k}(C_{f},\psi)$. 
Let $f_{0}(z) = \sum_{n=1}^{\infty} a(n,f_{0}) e(nz)$ be the Fourier expansion of $f_{0}$.
Comparing the Fourier coefficients in  \eqref{4.1}, we get $a(n,f_{0}) = a(n,f)-\alpha'(p)a(n/p,f)$. 
Hence,  we have the following identity for the corresponding Dirichlet series:
\begin{equation} \label{4.2}
	  L(s,f_{0})=  \sum\limits_{n=1}^{\infty} a(n,f_{0}) n^{-s} = (1- \alpha'(p)p^{-s}) 
	  \bigg(\sum\limits_{n=1}^{\infty}a(n,f)n^{-s} \bigg)= (1- \alpha'(p)p^{-s}) L(s,f).
\end{equation}
From \eqref{4.1}, it follows that $f_{0} \in S_{k}(pC_{f},\psi)$ and from
\eqref{2.8} and \eqref{4.2}, we have
\begin{equation}\label{4.3}
\begin{split}
	L(s,f_{0})  & = (1- \alpha'(p)p^{-s})\left( \prod\limits_{ q \ \mathrm{ prime}} 
	     (1- \alpha(q) q^{-s})^{-1}(1- \alpha'(q) q^{-s})^{-1}\right) \\
	& = ( 1- \alpha(p)p^{-s})^{-1} \Bigg( \sum\limits_{\stackrel{n=1}{p \nmid n}}^{\infty} a(n,f) n^{-s}\Bigg ).
	\end{split}
\end{equation}
Thus we have the following multiplicative relation
\begin{equation}\label{(4.5)}
	a(p^{r}n,f_{0})=\alpha(p^{r})a(n,f_{0}), \ \forall \ r, n \geq 0.
\end{equation} 
Hence, $f_{0}$ is a $U_{p}$-eigenvector with eigenvalue $\alpha(p)$, i.e., 
$f_{0}\vert U_{p}=\alpha(p) f_{0}$.

Recall $g \in S_l(N, \omega)$. From the definition of the operators $w_d$ and $V_{d}$ given in 
Section~\ref{operators}, one checks that
\begin{equation}\label{4.17}
g\vert w_{AB} = A^{l/2} g\vert w_{B}V_{A},
\end{equation}
where $A$, $B$ are positive integers.

Recall that complex valued Dirichlet characters $\chi$ of $p$-power conductor  
are the same as finite order characters $\chi: \mathbb{Z}_{p}^{\times}  \rightarrow \mathbb{C}^{\times}$.
As in Example $\ref*{example}$, we have $\mathcal{B} = \lbrace \chi:\mathbb{Z}_{p}^{\times}
\rightarrow \mathbb{C}^{\times} \mathrm{ \ of \ finite \ order}\rbrace $ 
forms a basis of Step$(\Zp^{\times}, \mathbb{C})$. Therefore every complex-valued function 
on $\mathcal{B}$ extends to a complex-valued  distribution on $\Zp^{\times}$.

Let $\chi : \mathbb{Z}_{p}^{\times}  \rightarrow \mathbb{C}^{\times}$
be a Dirichlet character of conductor $C_{\chi}$. 
Then $g(\chi) = \sum_{n=1 }^{\infty} \chi(n) b(n) e(nz)$ lies in $S_{l}(C_{g}C_{\chi}^{2},\omega \chi^{2})$, 
where here and below we use the convention that $\chi(n)=0$ if $n \in p\mathbb{Z}$.

For every $s \in \mathbb{C}$,  define a quantity 
$\Psi_{s}^{(M')}(\chi)$ as follows:
\begin{eqnarray}\label{4.6}
       \Psi_{s}^{(M')} (\chi) & = & \frac{(pM')^{s-l/2} C_{f}^{s-l/2}\overline{\chi}(C_{g})}
        { \Lambda(g) \alpha(pM')} \cdot 
        \frac{\Psi(s,f_{0}\vert V_{C_{f}},g(\chi)\vert w_{C_{0}M' })}{\pi^{1-l}
	\langle f , f \rangle_{C_{f}}},
\end{eqnarray} 
where $C_{0}$, $M'$ are  natural numbers satisfying:
\begin{align}\label{M'}
C_{0}=pC=pC_{f}C_{g}, \ p^{2}C_{\chi}^{2} \vert M' \  \mathrm{and} \ S(M') = \lbrace p \rbrace .  
\end{align}

A priori, the definition of $\Psi_{s}^{(M')}(\chi)$ depends on $M'$, though we show below that
it does not, whence $\Psi_{s}^{(M')}$ extends to a (well-defined) complex-valued distribution on 
$\mathbb{Z}_{p}^{\times}$. To do this, for each $s \in \C$, consider the complex-valued distribution 
$\Psi_{s}$ on $\mathbb{Z}_{p}^{\times}$ whose value on the Dirichlet character 
$\chi:  \mathbb{Z}_{p}^{\times}  \rightarrow \mathbb{C}^{\times}$ is given by:
\begin{eqnarray}\label{4.8}
      \Psi_{s}(\chi) & \coloneqq  & \frac{\omega(C_{\chi}) G(\chi)^{2} C_{\chi}^{2s-l-1}}
      {\alpha(C_{\chi})^{2}} \cdot \frac{\Psi(s,f,g^{\rho}(\overline{\chi}))}{\pi^{1-l}	
      \langle f , f 	\rangle_{C_{f}}}.
\end{eqnarray}

\begin{proposition}\label{Propn 4.1}
	 Let $p$ be an odd prime for which $f$ is a $p$-ordinary form.  
         Then for every Dirichlet character $\chi: \mathbb{Z}_{p}^{\times} \rightarrow \mathbb{C}^{\times}$ 
         and positive integer $M'$ such  that  
	 $p^{2}C_{\chi}^{2} \mid M'$ and  $S(M') = \lbrace p \rbrace$, we have
	\begin{align*}
	     \Psi_{s}^{(M')}(\chi) =  \Psi_{s}(\chi).
	\end{align*}
	In particular, $\Psi_{s}^{(M')}$ does not depend on $M'$.
\end{proposition}	
\begin{proof}
      First we simplify the right side of \eqref{4.6}.  From  \eqref{1.2} and \eqref{1.4} it follows that
      \begin{equation} \label{(4.7)}
            \begin{split}
             \Psi(s,f_{0}\vert V_{C_{f}},g(\chi)\vert w_{C_{0}M'}) 
              & =  (2\pi)^{-2s}\Gamma(s)\Gamma(s-l+1)L_{pC}(2s+2-k-l,\psi \overline{\omega\chi^{2}}) \\ 
                  & ~~~~~ \times  L(s,f_{0}\vert V_{C_{f}},g(\chi)\vert w_{C_{0}M'}),
             \end{split}
     \end{equation}
     noting that $S(pC) = S(C_0M')$, for the joint level $C_0M'$ of the forms $f_{0}\vert V_{C_{f}}$ and 
     $g(\chi)\vert w_{C_{0}M'}$.
     We define $A(n)$ and $B(n)$ to be the coefficients in the Dirichlet series
     \begin{equation}
           \begin{split}
                 L(s,f_{0}) & =  \sum\limits_{n=1}^{\infty} A(n)n^{-s},  \\
                 L(s,g(\chi) \vert w_{p^{2}C_{g}C_{\chi}^{2}}) & =\sum\limits_{n=1}^{\infty} B(n)n^{-s}  .
	     \end{split}
     \end{equation}
     Then, by the multiplicative property \eqref{(4.5)}, we have
     \begin{equation}\label{(4.9)}
          A(Mn)=\alpha(M)A(n),  \mathrm{ \ for \ all \ } M \mathrm{ \ such \ that \ }  
          S(M)= \lbrace p \rbrace.
     \end{equation}
     Let $M_{1}$ be such that $M'=pC_{\chi}^{2}M_{1}$. Applying \eqref{4.17} with  
     $A=M_{1}C_{f}$ and  $B=p^{2} C_{g} C_{\chi}^{2}$, we get
    \begin{equation}\label{(4.10)}
          \begin{split}
               g(\chi)\vert w_{C_{0}M'} = g(\chi)\vert w_{M_{1}C_{f}p^{2} C_{g} C_{\chi}^{2}} 
               & =  (M_{1}C_{f})^{l/2}  g(\chi) \vert w_{p^{2} C_{g} C_{\chi}^{2}} 
                      V_{M_{1}C_{f}}\\
               & =  (M_{1}C_{f})^{l/2}\sum\limits_{n=1}^{\infty} B(n)e(M_{1}C_{f}nz). 
          \end{split}
    \end{equation}
     %
     We transform the last $L$-function in 
     \eqref{(4.7)} as follows:
     \begingroup
     \allowdisplaybreaks
     \begin{eqnarray}\label{4.11}
           L(s,f_{0}\vert V_{C_{f}},g(\chi)\vert w_{C_{0}M'})  
           & \stackrel{\eqref{(4.10)}}{=}  & (M_{1} C_{f})^{l/2} \sum\limits_{n=1}^{\infty}A(nC_{f}^{-1}) B(nM_{1}^{-1}C_{f}^{-1})n^{-s}  \nonumber \\
           & =  & (M_{1} C_{f})^{l/2} \sum\limits_{n=1}^{\infty}A(nM_{1}) B(n)(nM_{1}C_{f})^{-s} \nonumber \\
           & \stackrel{\eqref{(4.9)}}{=} & (M_{1} C_{f})^{l/2-s} \alpha(M_{1})\sum\limits_{n=1}^{\infty}
               A(n) B(n)n^{-s} \nonumber \\
           & = & (M_{1} C_{f})^{l/2-s} \alpha(M_{1})L(s,f_{0},g(\chi) \vert w_{p^{2}C_{g}C_{\chi}^{2}}) \nonumber \\
           & = & \frac{\alpha(M')}{\alpha(pC_{\chi}^{2})} \cdot 
              \Big( \frac{M'C_{f}}{pC_{\chi}^{2}} \Big)^{l/2-s}  L(s,f_{0},g(\chi) \vert w_{p^{2}C_{g}C_{\chi}^{2}}). 
    \end{eqnarray}
    \endgroup
    If we substitute \eqref{4.11} in \eqref{4.6}, we see that  \eqref{4.6} does not depend on $M'$.  
    In order to obtain the more precise expression given by \eqref{4.8}, it is enough to establish the following equality:
    \begin{eqnarray}\label{4.12}
          \Psi (s,f_{0},g(\chi) \vert w_{p^{2}C_{g}C_{\chi}^{2}})  
          & = & \alpha(p)^{2} p^{l-2s} \Lambda(g(\chi))\Psi(s,f,g^{\rho}(\overline{\chi}))
    \end{eqnarray}
    where $\Lambda(g(\chi))$ is  the root number associated to $g(\chi)$, i.e., 
    $g(\chi) \vert w_{C_{g}C_{\chi}^{2}} = \Lambda(g(\chi))g^{\rho} (\overline{\chi})$,
    since by \eqref{rootnumber2} we have 
     $\Lambda(g(\chi))=\omega(C_{\chi})\chi(C_{g})G(\chi)^{2} C_{\chi}^{-1} \Lambda(g)$.
    
    To derive \eqref{4.12}  we find an appropriate expression  for $g(\chi)\vert w_{p^{2}C_{g}C_{\chi}^{2}}$. 
    Applying \eqref{4.17} once more with $A = p^{2}$ and $B = C_{g}C_{\chi}^{2}$, we get
    \begin{eqnarray*}\label{4.18}
         g(\chi)\vert w_{p^{2}C_{g}C_{\chi}^{2}} 
         & = & p^{l}  g(\chi) \vert w_{C_{g}C_{\chi}^{2}} V_{p^{2}} =  p^{l} \Lambda(g(\chi)) 
                    g^{\rho}(\overline{\chi})\vert V_{p^{2}}, 
    \end{eqnarray*}
    so that
    \begin{eqnarray*}
	    L(s,f_{0},g(\chi)\vert w_{p^{2}C_{g} C_{\chi}^{2}}) 
	    =  p^{l} \Lambda(g(\chi)) L(s,f_{0},g^{\rho}(\overline{\chi}) \vert V_{p^{2}}).
     \end{eqnarray*}
     A computation similar to  that of \eqref{4.11} shows that
    \begin{eqnarray*}
	      L(s,f_{0},g^{\rho}(\overline{\chi})\vert V_{p^{2}}) 
	           & = & p^{-2s} L(s,f_{0}\vert U_{p^{2}}, g^{\rho}(\overline{\chi})) \\
	           & = & \alpha(p^{2}) p^{-2s} L(s,f_{0},g^{\rho}(\overline{\chi})),
    \end{eqnarray*}
     where we used $f_{0} \vert U_{p} = \alpha(p) f_{0}$ in the last step. Therefore 
     \begin{eqnarray*} 
      \Psi (s,f_{0},g(\chi)
        \vert w_{p^{2}C_{g}C_{\chi}^{2}}) & = & \alpha(p)^{2} \Lambda(g(\chi)) \allowbreak  p^{l-2s} 
        \Psi(s,f_{0},g^{\rho}(\overline{\chi})). 
     \end{eqnarray*}
     Substituting this in \eqref{4.12} we are 
       reduced to proving
      \begin{eqnarray*}
	                  \Psi(s,f_{0},g^{\rho}  (\overline{\chi})) & = & \Psi(s,f,g^{\rho}  (\overline{\chi})).
     \end{eqnarray*}
      From \eqref{4.1}, it follows that 
      \begin{align}\label{4.19}
	         L(s,f_{0},g^{\rho}  (\overline{\chi})) 
	         & =   L(s,f,g^{\rho}  (\overline{\chi})) - \alpha'(p) L(s,f\vert V_{p},g^{\rho}  
	                  (\overline{\chi})) \nonumber \\
	         & = L(s,f,g^{\rho}  (\overline{\chi}))  \ \ \ \  (\because   \chi(p)=0).
    \end{align}
     Further, for every character $\chi: \mathbb{Z}_{p}^{\times} \rightarrow \mathbb{C}^{\times}$, 
     we  have $S(pC_fC_gC_\chi^2) = S(C C_{\chi}^2)$ (except if $\chi$ is the trivial character) so that 
     \begin{eqnarray}\label{4.20}
         L_{pC_fC_gC_\chi^2}(2s+2-k-l,\psi\overline{\omega\chi^{2}}) 
            & = &  L_{CC_{\chi}^{2}}(2s+2-k-l,\psi\overline{\omega\chi^{2}})
     \end{eqnarray}
     in all cases (since if $\chi$ is the trivial character, $\chi(p) = 0$). 
    From \eqref{4.19} and \eqref{4.20}, it follows that $\Psi(s,f_{0},g^{\rho}  (\overline{\chi}))
     = \Psi(s,f,g^{\rho}  (\overline{\chi}))$. Thus we obtain \eqref{4.12}. 
\end{proof}
We conclude this section by making an observation on the algebraicity of $\Psi_{s}^{(M')}$, 
which will be used in later sections.
\begin{corollary}\label{Psi is algebraic}
	Let $\chi: \mathbb{Z}_{p}^{\times} \rightarrow \mathbb{C}^{\times}$ be a finite order character 
	and $M'$ as in Proposition \ref{Propn 4.1}. Then for every integer $s$ with $l \leq s \leq k-1$, we 
	have $\Psi_{s}^{(M')}(\chi) \in \Qbar$.  
\end{corollary}
\begin{proof}
       From  Theorem \ref{Algebraicity}, we have $\Psi_{s}(\chi)$ is algebraic for every integer $s$ 
       with $l \leq s  \leq k-1$. Hence, by the previous  proposition, we have $\Psi_{s}^{(M')}$ is also   
       algebraic for every integer $s$ in the interval $[l ,k-1]$.
 \end{proof}
Dirichlet characters actually take values in $\overline{\Q} \subset \C$. Via our fixed
embedding $i_p \> : \> \overline{\Q} \hookrightarrow \mathbb{C}_p$, we may think of them as $\C_p$-valued. 
Moreover, by the corollary above we may similarly think of  $\Psi_{s}^{(M')}$ as $\C_p$-valued for 
$s \in [l,k-1]$. Thus, for such $s$, all the measures in this section can (and later will be) thought of as 
$p$-adic entities. 
 
\section{Integral representation for Distributions}

In this section we obtain an integral expression for the distribution $\Psi^{(M')}_{s}$ given by \eqref{4.6}
involving the Petersson inner product of certain cusp forms. We also compute
the Fourier expansion of one of these cusp forms. This will be needed
in the last section in order to explicitly verify the Kummer congruences.  

Recall the following classical integral formula of Rankin (cf. \eqref{2.12}). 
For $F \in S_{k}(N,\psi)$ and $G \in M_{l}(N,\omega) $, we have 
\begin{equation}\label{5.1}
	  \Psi(s,F,G)= 2^{-1}\Gamma(s-l+1)\pi^{-s} \langle F^{\rho},GE(s-k+1)\rangle_{N}, 
\end{equation}
where 
\begin{align*}
	F^{\rho}(z) &  = \overline{F(- \overline{z})} \ \in S_{k}(N,\overline{\psi} ), \\
	E(z,s)         &  =  E_{k-l,N}(z,s,\psi\omega) = y^{s} \sideset{}{'} \sum\limits_{c,d = - \infty} ^ {\infty}  
	                            \psi\omega(d) (cNz+d)^{-(k-l)} \vert cNz+d \vert ^{-2s}.
\end{align*}
Let $\chi: \mathbb{Z}_{p}^{\times} \rightarrow \mathbb{C}_{p}^{\times}$  be a finite order character.
 Let $M'$ be as in \eqref{M'}, i.e., $p^{2}C_{\chi}^{2}  \mid  M'$ and $S(M')=\lbrace p \rbrace$. 
 We apply \eqref{5.1} with
\begin{align*}
       N &=  C_{0}C_{f}M' , \\
    	   F &= f_{0} \vert V_{C_{f}} \in S_{k}(pC_{f}^{2},\psi) \subset S_{k}(C_{0}C_{f}M',\psi),   \\
	    G  & = g(\chi) \vert  w_{C_{0}M'}  \in S_{l}(C_{0}M',\overline{\omega\chi^{2}}) 
	                \subset S_{l}(C_{0}C_{f}M',\overline{\omega\chi^{2}}).	
\end{align*}
For  every integer $s$ such that $l \leq s \leq k-1$, we transform the definition of the distribution 
\eqref{4.6} by means of the equality
 \begin{eqnarray*}
	     \Psi(s,f_{0}\vert V_{C_{f}},g(\chi)\vert w_{C_{0}M'}) 
	     & = & 2^{-1}\Gamma(s-l+1)\pi^{-s} \langle f_{0}^{\rho}\vert V_{C_{f}},GE(s-k+1)\rangle_{C_{0}C_{f}M'},
\end{eqnarray*}
where $E(z,s-k+1) = E_{k-l,C_{0}C_{f}M'}(z,s-k+1,\psi\overline{\omega\chi^{2}}).$ 
If we set
\begin{align*}
    	K(s) =G \cdot E(z,s),
\end{align*}
then the formula for the values of the distribution \eqref{4.6}  takes the form
\begin{equation}
      \begin{split}
	       \Psi_{s}^{(M')}(\chi) = (pM')^{s-l/2}& C_{f}^{s-l/2} \overline{\chi}(C_{g}) \Lambda(g)^{-1} 
	                                               \alpha(pM')^{-1} \\ &
            \times  2^{-1}  \Gamma(s-l+1)  \pi^{-s} \frac{\langle f_{0}^{\rho} \vert V_{C_{f}} , K(s-k+1) 
              \rangle_{C_{0}C_{f}M'}}{\pi^{1-l} \langle f,f \rangle_{C_{f}}}.
	\end{split}
\end{equation}
By Lemma~\ref{lemma2.2} (with $N=C_{0}C_{f}$, $M=M'$, $f=f_{0}^{\rho}\vert V_{C_{f}}$ and 
$g=K(s)$), we obtain 
\begin{eqnarray*}
     	\langle  f_{0}^{\rho} \vert V_{C_{f}}, K(s) \rangle_{C_{0}C_{f}M'} 
        & =  & \langle f_{0}^{\rho}\vert V_{C_{f}}, Tr_{C_{0}C_{f}}^{C_{0}C_{f}M'}(K(s))\rangle_{C_{0}C_{f}} \\
      	& \stackrel{\eqref{2.10}}{=}  & (-1)^{k} M'^{1-k/2} \langle f_{0}^{\rho} \vert V_{C_{f}}, K'(s) \vert U_{M'} 
      	   w_{C_{0}C_{f}} \rangle_{C_{0}C_{f}},
\end{eqnarray*}
where $K'(s) = K(s) \vert w_{C_{0}C_{f}M'}$. Hence,
\begin{equation}\label{5.5}
       \begin{split}
               \Psi_{s}^{(M')}(\chi) = (-1)^{k} & p^{s-l/2}  M'^{(2s-l-k+2)/2} C_{f}^{s-l/2}
               \overline{\chi} (C_{g})  \Lambda(g)^{-1} \alpha(pM')^{-1} \\ &
                \times	2^{-1}  \Gamma(s-l+1)  \pi^{-s} \frac{\langle f_{0}^{\rho}\vert V_{C_{f}} , 
                K'(s-k+1)\vert U_{M'}w_{C_{0}C_{f}} \rangle_{C_{0}C_{f}}} {\pi^{1-l} \langle f,f \rangle_{C_{f}}}.
        \end{split}
\end{equation} 
Now we compute the Fourier coefficients of $K'(s)$ for special values of $s$ (more precisely, for 
$l-k+1 \leq s \leq 0$, $s \in \mathbb{Z}$). We rewrite $K'(s)$ as 
\begin{align*}
      K'(s)= g' \cdot E'(z,s),
\end{align*}
where
\begin{align*}\label{5.6}
	g'=g(\chi)\vert w_{C_{0}M'} w_{C_{0}C_{f}M'} \ \ \mathrm{and} \ \ E'(z,s)=E(z,s) \vert w_{C_{0}C_{f}M'}.
\end{align*} 
It follows from the definition of $w_{C_{0}M'}$, $w_{C_{0}C_{f}M'}$ that 
\begin{equation}\label{5.7}
	    g'=(-1)^{l}C_{f} ^{l/2}g(\chi) \vert V_{C_{f}}. 
\end{equation}
The Fourier expansion of the Eisenstein series   $E'(z,s)$ will be computed in the next section, 
from which we will obtain the Fourier expansion of $K'(s)$.
\subsection{Fourier expansion of Eisenstein series}
Here we follow \cite[$ \S 7.2$]{Miyake} to compute the Fourier expansion of $E'(z,s)$. 
The procedure given in \cite{Miyake} describes the Fourier expansion of more general Eisenstein series. 
Let $\mathcal{H}' = \lbrace z \in \mathbb{C} \mid \mathrm{Re}(z)>0 \rbrace$ denote the right half plane. 
For $\alpha \in \mathbb{C}$ and 
$\beta, z \in \mathcal{H}'$, the Whittaker function  $W(z;\alpha,\beta)$ is defined by the following integral:
\begin{equation}
	   W(z;\alpha,\beta)=\Gamma(\beta)^{-1} \int\limits_{0}^{\infty} (u+1)^{\alpha-1} u^{\beta-1}e^{-zu} \ du.
\end{equation}
The convergence of the above integral follows from \cite[Lemma 7.2.1 (2)]{Miyake}.
\begin{lemma}\label{(5.1)}
	The function $W(z;\alpha,\beta)$ can be continued analytically to a holomorphic function on 
	 $\mathcal{H}' \times \mathbb{C} \times \mathbb{C}$  satisfying:
	 \begin{enumerate}[label={\emph{(\arabic*)}}]
	        \item  $W(z;\alpha,\beta)= z^{1-\alpha-\beta}W(z;1-\beta,1-\alpha), \ \forall \ (z,\alpha,\beta) 
	                   \in  \mathcal{H}' \times \mathbb{C} \times \mathbb{C}$.
	        \item   $W(z;\alpha,0) = 1, \ \forall \ (z,\alpha) \in \mathcal{H}' \times \mathbb{C}$.
          	\item  
                       $W(y;\alpha,\beta) = \sum\limits_{i=0}^{r}(-1)^{i} 
        	               \binom{r}{i} y^{r-i} \frac{\Gamma(\alpha)} {\Gamma(\alpha-i)}W(y;\alpha - i,\beta + r), 
                       \> \forall \ r \geq 0$,  $y \in \mathbb{R}^{+}$,  $(\alpha,\beta) \in \mathbb{C}\times \mathbb{C}$. 
	\end{enumerate}
\end{lemma}
\begin{proof}
	Note that $\omega(z;\alpha,\beta)$ defined by \cite[(7.2.31)]{Miyake} equals to   
	$ z^{\beta}W(z;\alpha,\beta)$ for all $(z,\alpha,\beta) \in \mathcal{H}' \times \mathbb{C} \times 
	\mathcal{H}'$. The lemma now follows from \cite[Theorem 7.2.4 (1)]{Miyake}, 
	\cite[Lemma 7.2.6]{Miyake} and \cite[(7.2.40)]{Miyake}.
\end{proof}
By part (3) of Lemma~\ref*{(5.1)}, with $\beta =-r$, and by part (2), we obtain for all $y>0$ that 
\begin{equation}\label{5.8}
	\begin{split}
	       W(y;\alpha,-r) 
	       &  =  \sum\limits_{i=0}^{r} (-1)^{i} \binom{r}{i} \frac{\Gamma(\alpha)}{\Gamma(\alpha-i)} 
	                 y^{r-i} W(y;\alpha-i,0) \\
	       & =  \sum\limits_{i=0}^{r} (-1)^{i} \binom{r}{i} \frac{\Gamma(\alpha)}{\Gamma(\alpha-i)} y^{r-i}.
	\end{split}
\end{equation}
Recall that the Eisenstein series $E_{k}(z,s;\theta,\varphi)$ for $\theta$ and $\varphi$ Dirichlet 
characters mod $L$ and $M$ respectively is defined by (cf. \eqref{Eisenstein series})
\begin{equation*}
	E_{k}(z,s;\theta,\varphi) = y^{s} \sideset{}{'} \sum\limits_{c,d = - \infty} ^ {\infty} 
	\theta(c) \varphi(d) (cz+d)^{-k} \vert cz+d \vert ^{-2s} .
\end{equation*}
We now state a result  about the Fourier expansion of  Eisenstein series.
\begin{theorem}\label{thm5.2}
	Let $\theta$ and $\varphi$ be Dirichlet characters mod $L$ and mod $M$, respectively, satisfying
	$\theta(-1)\varphi(-1) = (-1)^{k}$. Then for any integer $k$, the Eisenstein series $E_{k}(z,s;\theta,\varphi)$ 
	can be analytically continued to a meromorphic function on the whole $s$-plane and has the Fourier expansion 
	\begin{align*}
		 E_{k}(z,s;\theta,\varphi)= \ &  C(s) y^{s}+D(s)y^{1-k-s} +
		  A(s)y^{s}\sum\limits_{n=1}^{\infty} a_{n}(s)(4\pi /M)^{s}e( n z/M) W(4\pi y n /M;k+s,s) \\
		 &+B(s)y^{s}\sum\limits_{n=1}^{\infty} a_{n}(s)(4 \pi / M)^{s+k}e(- n \bar{z}/M) W(4\pi y n /M;s,k+s),
	\end{align*}
	where 
	\begin{eqnarray*}
		C(s) & = & \begin{cases}  
                               2L_{M}(2s+k,\varphi), & \hspace{3mm} \mathrm{if} \  \theta = \chi_{0},\\
		                                  0, & \hspace*{3 mm} \mathrm{otherwise},  
                           \end{cases} \\
	%
	 	  D(s) & =  &
                  \begin{cases}
		  2  \sqrt{\pi}i^{-k}\prod\limits_{p \mid M} (1-p^{-1}) \Gamma(s)^{-1} \Gamma(s+k)^{-1} & \\
		   \quad \times \> \Gamma \left( \frac{2s+k-1}{2} \right) \Gamma \Big (\frac{2s+k}{2} \Big ) L_{L}(2s+k-1,\theta), 
		   & \mathrm{if} \  \varphi   \ \mathrm{is \ the \ trivial \ character \ mod} \ M, \\
	 	    0, & \mathrm{otherwise}, 
	 	 \end{cases} \\
	 %
		  A(s) &= & 2^{k+1}i^{-k}G(\varphi^{0})(\pi \text{/} M)^{s+k} \Gamma(s+k)^{-1}, \hspace{2.2 cm} \\
	 	  B(s) &= & 2^{1-k}i^{-k}\varphi(-1)G(\varphi^{0})(\pi/M)^{s} \Gamma(s)^{-1},\hspace{1.3 cm} \\
	  	  a_{n}(s) &= & \sum\limits_{0 < c \vert n} \theta(n/c) c^{k+2s-1} \sum \limits_{0 < d \mid (l,c)} d  \mu(l/d) 
	      \varphi^{0}(l/d)  \overline{\varphi^{0}}(c/d).
	  \end{eqnarray*}
	  Here $\varphi^{0}$ denotes the primitive character associated with $\varphi$ of conductor 
	  $m_{\varphi}= M/l$ and $\mu$ is the M$\ddot{o}$bius function.		
\end{theorem}	
	\begin{proof}
           This is \cite[Theorem 7.2.9]{Miyake},
	   noting that $E_{k}(z,s;\theta,\varphi)$ differs from the one defined in \cite[(7.2.1)]{Miyake} by a
	    factor of $y^{s}$ and $\omega(y;\alpha,\beta)$ equals $y^{\beta}W(y;\alpha,\beta), \ \forall \ (\alpha,\beta) 
	    \in \mathbb{C}\times \mathbb{C}$. 
    \end{proof}	

We apply the above theorem  to compute the Fourier expansion of $E'(z,s)$. Recall
\begin{eqnarray}\label{E' and Ek}
	E'(z,s) 
	& = &E(z,s) \vert w_{C_{0}C_{f}M'} 
	\ \ =  \ \ E_{k-l,C_{0}C_{f}M'}(z,s,\psi\overline{\omega \chi^{2}}) \vert w_{C_{0}C_{f}M'} \nonumber \\
	&= & (C_{0}C_{f}M')^{-(k-l+2s)/2}E_{k-l}(z,s;\psi\overline{\omega \chi^{2}},\chi_{0}) 
	~~~\mathrm{ \ (by  \ direct \ computation).}
\end{eqnarray}
For convenience  we introduce the normalized Eisenstein series
\begin{align}
	G^{*}(z,s) \ \ & = \ \ \frac{(C_{0}C_{f}M')^{(k-l+2s)/2} \Gamma(k-l+s)} {(-2\pi i)^{k-l}\pi^{s}} E'(z,s) \label{G* and E'} \\
	&    \stackrel{\mathclap{\eqref{E' and Ek}}}{=} \ \  \frac{ \Gamma(k-l+s)} {(-2\pi i)^{k-l} \pi ^{s}}
	     E_{k-l}(z,s;\psi\overline{\omega \chi^{2}},\chi_{0}) \label{G* and Ek}.
\end{align}
If $s$ is an integer such that $s \leq 0$ and $k-l+s > 0$, then from \eqref{G* and Ek} and Theorem~\ref{thm5.2}, we 
have 
\begin{align}\label{5.10}
	      G^{*}(z,s) & =   \varepsilon(k-l,y,s,\psi\overline{\omega \chi^{2}})+2 (4 \pi y)^{s} 
	       \sum\limits_{n=1}^{\infty} \sum_{0 < c \vert n} \psi\overline{\omega \chi^{2}}(n/c)  
	        c^{k-l+2s-1}W(4\pi y n;k-l+s,s)e(nz) \nonumber \\
	       & = \varepsilon(k-l,y,s,\psi\overline{\omega \chi^{2}})+2(4\pi y)^{s}
	        \sum\limits_{n=1}^{\infty} \sum_{dd'= n} \psi\overline{\omega \chi^{2}}(d')  
	        d^{k-l+2s-1}W(4\pi y n;k-l+s,s) e(nz),
\end{align}
where 
$$\varepsilon(k-l,y,s,\psi\overline{\omega \chi^{2}}) = 
\frac{ \Gamma(k-l+s)} {(-2\pi i) ^{k-l}\pi^{s}}(C(s) y^{s}+D(s)y^{1-k+l-s}),$$ 
with $C(s)$, $D(s)$ denoting the same constants as in Theorem~\ref{thm5.2} (corresponding to 
$\theta = \psi\overline{\omega \chi^{2}}$, $\varphi= \chi_{0}$). The  term with $\overline{z}$ 
doesn't appear as for such $s$ we have $B(s)=0$ because the Gamma function $\Gamma(s)$ in the denominator 
of $B(s)$ has a pole at $s\leq 0$ and the function $a_n(s)W(4\pi yn/M;k+s,s)$ is holomorphic in $s$. 

\subsection{Integral representation via holomorphic projection}

Taking $s$ equal to $s-k+1 \leq 0$ in \eqref{G* and E'}, we get\footnote{This formula differs from \cite[(4.22)]{Panrankin}  by $(-1)^{s-k+1}$ 
and is
the source of the sign discrepancy in Theorem~\ref{maintheorem} mentioned in the first footnote.  
Without the sign in \eqref{5.11}, it becomes difficult to verify the abstract Kummer congruences in the proof of 
Proposition~\ref{prop Kummer} (2) later.}
\begin{eqnarray}\label{5.11}
             E'(z,s-k+1) & = & (C_{0}C_{f}M')^{-(2s+2-k-l)/2}(-1)^{k-l} i^{k-l}  2^{k-l} \pi^{s-l+1} \nonumber \\
                         &   &  \quad \times \> \Gamma(s-l+1)^{-1} G^{*}(z,s-k+1).  
\end{eqnarray}
Substituting \eqref{5.11} and \eqref{5.7} into \eqref{5.5}, and substituting $C_{0} = pC = pC_{f}C_{g}$, 
we get  
\begin{align}\label{5.13}
      \Psi_{s}^{(M')}(\chi) &= \frac{ 2^{k-l-1} i^{k-l}p^{k/2 -1} \overline{\chi}(C_{g}) C_{f}^{(k+l-2)/2} }  
       { \alpha (pM')  \Lambda(g)  C^{(2s+2-k-l)/2}} \cdot \frac{\langle   f_{0}^{\rho} \vert V_{C_{f}} , 
       (g(\chi) \vert V_{C_{f}} G^{*}(z,s-k+1))\vert U_{M'}w_{C_{0}C_{f}} \rangle_{C_{0}C_{f}}}
        {\langle f , f \rangle_{C_{f}}} \nonumber \\
        &= \gamma(M') \langle f , f \rangle_ {C_{f}}^{-1}  \ \langle   f_{0}^{\rho} \vert _{V_{C_{f}}} ,K^{*}(s-k+1)
         \vert U_{M'}w_{C_{0}C_{f}} \rangle_{C_{0}C_{f}},
\end{align}
in which we have set 
\begin{equation}\label{5.12}
      \begin{split}
            & \gamma(M') = 2^{k-l-1} i^{k-l} p^{k/2-1} C_{f}^{l-1}C_{g}^{(l-k)/2}   \alpha(pM')^{-1}
                 \Lambda(g)^{-1}, \\
            &   K^{*}(s) = C_{f}^{-s}C_{g}^{-s} \overline{\chi}(C_{g}) g(\chi)\vert V_{C_{f}} G^{*}(z,s).
     \end{split}
\end{equation}

Observe that $\gamma(M')$ is an algebraic number. Moreover, $i_{p}(\gamma(M'))$ is $p$-integral if $i_p(\Lambda(g))$ is a $p$-adic unit.
One can check this last fact using explicit formulas for the root number in terms of Gauss sums when the automorphic representation attached to $g$ has 
no supercuspidal local factors; it is apparently also true in general \cite[(5.4a), (5.4b)]{Hid88}. In any case $i_{p}(\gamma(M'))$ is bounded
independent of $M'$, which is all we shall need later.  

It follows 
from  \eqref{5.10} that for integers $l-k < s\leq 0$ we have 
\begin{equation}\label{K*(s) Fourier series}
      K^{\ast}(s) = \sum\limits_{n=1}^{\infty} \sum\limits_{C_{f}n_{1}+n_{2} =n} d(n_{1},n_{2};y,s)e(nz),
\end{equation}
where for $p \mid n$, the Fourier coefficients are given by\footnote{The formula differs from \cite[(4.27)]{Panrankin} by $(-1)^s$ due to the sign error mentioned in the previous footnote.}
\begin{equation}\label{Fourier coefficients of K*(s)}
      \begin{split}
	        d(n_{1},n_{2};y,s) =   C_{f}^{-s} & C_{g}^{-s} \overline{\chi}(C_{g}) \chi(n_{1}) b(n_{1}) \\ 
	             & \times 2 (4 \pi y)^{s} \sum\limits_{n_{2}=dd'} \psi\overline{\omega \chi^{2}}(d') 
	                 d^{2s+k-l-1}W(4\pi n_{2}y,s-l+k,s).
      \end{split}
\end{equation} 
Here we used that if $p\mid n$ there is no contribution to the coefficient of  $e(nz)$ in $K^{\ast}(s)$
 from the constant ($n_2 = 0$) term of Eisenstein series $G^{\ast}(z,s)$ because the coefficient of 
 $e(C_fn_1z)$ in $g(\chi)\vert V_{C_{f}}$ is zero for $p \mid n_1$ since $\chi(n_1) = 0$.  

The expression  \eqref{5.13}  for $\Psi_{s}^{(M')}(\chi)$ involves $K^*(s-k+1)$ whose Fourier 
coefficients contain Whittaker functions which are difficult to handle. To get rid of the Whittaker functions 
we consider its holomorphic  projection. We first check that  $\mathcal{H}ol(K^*(s-k+1))$ is defined.  
 From Proposition~\ref{prop2.9}, it follows that  if $(k+l)/2 < s \leq k-1$, then
$E_{k-l}(z,s-k+1; \psi \overline{ \omega \chi^{2}}, \chi_{0})$ belongs to
 $\mathcal{N}_{k-l}^{-s+k-1}(C_{0}C_{f}M',\psi\overline{ \omega\chi^{2}})$, hence so does 
 $G^*(z,s-k+1)$, by \eqref{G* and Ek}. Thus
 $K^{*}(s-k+1) \in \mathcal{N}S_{k}^{-s+k-1}(C_{0}C_{f}M',\psi)$ if $s>(k+l)/2$. So for such 
 $s$ one can define the holomorphic  projection $\mathcal{H}ol(K^{*}(s-k+1))$ of 
 $K^{*}(s-k+1)$ in the sense of Theorem~\ref{thm2.8}.
 However, for  $l \leq s \leq (k+l)/2 $ it is  not clear (to us) that
$K^{*}(s-k+1)$ is a nearly holomorphic form. 
So we cannot use Theorem~\ref{thm2.8} to define
the holomorphic projection  of $K^{*}(s-k+1)$ for $l \leq s \leq (k+l)/2$.
Nevertheless, by the discussion at the end of $\S$2, we
know that $K^{*}(s)$ is rapidly decreasing and satisfies the hypotheses
of Lemma \ref{hol}, with $c_{0}=0$. Thus one can define the holomorphic  projection of
$K^{\ast}(s-k+1)$ for any integer $l \leq s \leq k-1$. 

We now study:
\begin{equation*}
	\widetilde{K}_{M'}(s) \coloneqq \mathcal{H}ol(K^{*}(s)) \vert U_{M'},
\end{equation*}
for integers $l-k+1 \leq s \leq 0$.
We begin by  computing the level and nebentypus of $\widetilde{K}_{M'}(s)$.  Since  
$K^{\ast}(s) \vert_{k} \gamma = \psi(\gamma) K^{\ast}(s)$ for all 
$\gamma \in \Gamma_{0}(C_{0}C_{f}M')$,  we have 
$\mathcal{H}ol(K^{*}(s)) \in S_{k}(C_{0}C_{f}M',\psi)$, by the remarks after Lemma~\ref{hol}. As $ p^{2} \mid M'$ we have 
$\mathcal{H}ol(K^{*}(s)) \vert U_{p} \in S_{k}(C_{0}C_{f}M'/p, \psi)$, by Lemma  \ref{prop2.1} (1). 
Repeatedly applying Lemma~\ref{prop2.1} (1) we get   
$\mathcal{H}ol(K^{*}(s)) \vert U_{M'}  \in S_{k}(C_{0}C_{f},\psi)$. 

 We now state the main result of this 
section.
\begin{proposition}\label{propn5.3}
	   Let the notation be as above. For $s \in \mathbb{Z}$ with $l \leq s \leq k-1$ one has following equality
	   \begin{equation}\label{final expression of measure}
		      \Psi_{s}^{(M')}(\chi) = \gamma(M') \langle f,f\rangle^{-1}_{C_{f}} \langle 
		      f_{0}^{\rho}\vert V_{C_{f}},\widetilde{K}_{M'}(s-k+1)\vert w_{C_{0}C_{f}} \rangle_{C_{0}C_{f}}.
	 \end{equation}
	 Moreover, for $s \in \mathbb{Z}$ with $l-k+1 \leq s \leq 0$ we have
	 \begin{equation}\label{first expression for K_M'}
	         \widetilde{K}_{M'}(s) = \sum\limits_{n=1}^{\infty} \sum\limits_{C_{f}n_{1}+n_{2} =M'n} 
	          d(n_{1},n_{2};s,\chi)e(nz)  \in  S_{k}(C_{0}C_{f},\psi)
	 \end{equation}
	is a cusp form with algebraic Fourier coefficients given by\footnote{The formula differs from \cite[(4.29)]{Panrankin} by the same sign
as in the previous footnote.} 
	\begin{equation}\label{5.16}
		  d(n_{1},n_{2};s,\chi) =  2  C_{f}^{-s}  C_{g}^{-s} \overline{\chi}(C_{g}) \chi(n_{1}) b(n_{1}) 
		   \sum\limits_{n_{2}=dd'} \psi\overline{\omega \chi^{2}}(d') d^{2s+k-l-1} P_{s}(n_{2},M' n)
      \end{equation} 
	and 
	\begin{equation}\label{5.17}
	       \begin{split}
		         P_{s}(x,y) 
		         & = \sum\limits_{i=0}^{-s} (-1)^{i} \binom{-s}{i} \frac{\Gamma(s+k-l)\Gamma(k-i-1)}
		           {\Gamma(s+k-l-i) \Gamma(k-1) } x^{-s-i} y^{i} \\
		       &=x^{-s}+\frac{y}{\Gamma(k-1)}Q_{s}(x,y), \  \mathrm{where \ } s \leq 0 \mathrm{ ~and} 
		          ~Q_{s}(x,y)   \in \mathbb{Z}[x,y]. 
		\end{split}
	\end{equation}
\end{proposition}	
\begin{proof}
	  The  proof of the lemma is an application of the holomorphic projection lemma (Lemma \ref{hol}).
	  %
	  %
        We first note that $\mathcal{H}ol$ commutes with the action of the $w_N$-operator. Indeed, by Lemma~\ref{hol} and \eqref{2.2}, we have
$$\langle h, \mathcal{H}ol(\Phi \vert w_N) \rangle_N =  \langle h, \Phi \vert w_N \rangle_N = \langle h \vert w_N, \Phi \rangle_N
  = \langle h \vert w_N,  \mathcal{H}ol(\Phi) \rangle_N = \langle h, \mathcal{H}ol(\Phi) \vert w_N \rangle_N, $$
for all modular rapidly decreasing $\Phi$ and all cusp forms $h$ of weight $k$ and level $N \geq 1$, 
 whence  
$\mathcal{H}ol(\Phi \vert w_N) =  \mathcal{H}ol(\Phi) \vert w_N$. A similar argument shows that  $\mathcal{H}ol$ commutes with the 
$U_p$-operator. Thus, by Lemma~\ref{lemma2.10} and Lemma~\ref{hol}, we have
	 \begingroup
	 \allowdisplaybreaks
	 \begin{eqnarray*}
            \langle f_{0}^{\rho} \vert V_{C_{f}} ,K^{*}(s-k+1) \vert U_{M'}w_{C_{0}C_{f}} \rangle_{C_{0}C_{f}}
             & = & \langle f_{0}^{\rho} \vert V_{C_{f}} , 
                    \mathcal{H}ol(K^{*}(s-k+1) \vert U_{M'}w_{C_{0}C_{f}})  \rangle_{C_{0}C_{f}} \\ 
        	    &  = & \langle  f_{0}^{\rho} \vert V_{C_{f}} , 
        	            \mathcal{H}ol(K^{*}(s-k+1)) \vert U_{M'}w_{C_{0}C_{f}}  \rangle_{C_{0}C_{f}} \\
	        &  = &  \langle f_{0}^{\rho} \vert V_{C_{f}} ,
	             \widetilde{K}_{M'}(s-k+1) \vert w_{C_{0}C_{f}} \rangle_{C_{0}C_{f}}.
	\end{eqnarray*}
	\endgroup
	Substituting the above expression in \eqref{5.13}, we obtain \eqref{final expression of measure}.
        It follows from \eqref{K*(s) Fourier series}, \eqref{Fourier coefficients of K*(s)} that
	\begingroup
	\allowdisplaybreaks
	\begin{align}\label{5.21}
	      K^{*}(s) \vert U_{M'} & = ~~ M'^{k/2-1} \sum\limits_{u \ \mathrm{mod} \ M' } K^{*} (s) \vert 
	         \begin{pmatrix} 1 & u \\ 0 & M' \end{pmatrix} \nonumber \\
	       & \stackrel{\mathclap{\eqref{K*(s) Fourier series}}}{=} ~~  M'^{-1} \sum\limits_{u \ \mathrm{mod} \ M' } 
	       \sum\limits_{n=1}^{\infty} \sum\limits_{C_{f}n_{1}+n_{2} =n}  d(n_{1},n_{2};y/M',s)  e(n(z+u)/M') 
	        \nonumber \\
	       &= ~~\sum\limits_{n=1}^{\infty} \sum\limits_{C_{f}n_{1}+n_{2} =n}  d(n_{1},n_{2};y/M',s) e(nz/M') 
	           M'^{-1} \sum\limits_{u \ \mathrm{mod} \ M' } e(un/M')  \nonumber \\
           &=~~\sum\limits_{n=1}^{\infty} \sum\limits_{C_{f}n_{1}+n_{2} =M'n} d(n_{1},n_{2};y/M',s)e(nz).
	\end{align}
	\endgroup

	Now we use Lemma \ref*{hol} to compute the  Fourier coefficients of $\widetilde{K}_{M'}(s-k+1)
	=\mathcal{H}ol(K^{*}(s-k+1) \vert U_{M'})$ for $l \leq s \leq k-1$. Let $s'=s-k+1$ then $l-k+1 \leq s' \leq 0$. 
	 From \eqref{5.21} and Lemma \ref{hol} it follows that
	\begin{align}\label{5.19}
	     	\widetilde{K}_{M'}(s')  = \sum\limits_{n=1}^{\infty} \sum\limits_{C_{f}n_{1}+n_{2} =M'n } 
	     	\frac{(4\pi n )^{k-1}} {\Gamma(k-1)}\Big ( \int \limits_{0}^{\infty} d(n_{1},n_{2};y/M',s') 
	     	 e^{-2\pi n y} e^{-2 \pi n y}y^{k-2} \ dy \Big ) e(nz).
	\end{align}
        Note that if $C_{f}n_{1}+n_{2}=M'n$, the quantity $d(n_{1},n_{2};y/M',s)$ is as in \eqref{Fourier coefficients of K*(s)}, with
        $y$ replaced by $y/M'$, because 
	$p \mid M'n$, since $p \mid M'$. We get
        \begin{eqnarray}
        \label{5.24}
        d(n_{1},n_{2};s',\chi) & \coloneqq & \frac{(4 \pi n)^{k-1}}{\Gamma(k-1)} 
	\int_{0}^{\infty}d(n_{1},n_{2};y/M',s')e^{-4 \pi ny} y^{k-2}  dy, \nonumber \\   
	%
	      & = & 2 (C_{f}C_{g})^{-s'} \overline{\chi}(C_{g}) \chi(n_{1}) b(n_{1})\sum\limits_{n_{2}=dd'} 
	         \psi\overline{ \omega \chi^{2}}(d') d^{2s'+k-l-1}  \nonumber \\
     	      &   &  \times\frac{(4\pi n)^{k-1}}{\Gamma(k-1)} \int\limits_{0}^{\infty}
 	     \Big ( \frac{4\pi y}{M'} \Big )^{s'} W\left(\frac{4\pi n_{2}y}{M'},s'+k-l,s'\right)e^{-4 \pi ny} y^{k-2} dy . 
       \end{eqnarray}
	Since $l-k+1 \leq s' \leq 0$, we can use \eqref{5.8} to compute $W(4\pi n_{2}y/M',s'+k-l,s')$. We obtain
    %
	 \begin{align*}
	   	&  \frac{(4\pi n)^{k-1}}{\Gamma(k-1)} \int\limits_{0}^{\infty}\Big (\frac{4 \pi y}{M'} \Big )^{s'} 
	   	     W \Big (\frac{4\pi n_{2}y}{M'},s'+k-l,s' \Big )e^{-4 \pi ny} y^{k-2} dy\\
		& \hspace{1 cm} = \sum\limits_{i=0}^{-s'} (-1)^{i} \binom{-s'}{i}\frac{\Gamma(s'+k-l)} 
		     {\Gamma(s'+k-l-i)\Gamma(k-1)}\int\limits_{0}^{\infty}(4 \pi n)^{k-1} \Big ( \frac{4\pi y}{M'}\Big )^{s'}
		       \Big ( \frac{4 \pi n_{2}y}{M'} \Big )^{-s'-i}e^{-4 \pi ny} y^{k-2} dy \\
		 & \hspace{1 cm} = \sum\limits_{i=0}^{-s'} (-1)^{i} \binom{-s'}{i} 
		      \frac{\Gamma(s'+k-l)} {\Gamma(s'+k-l-i)\Gamma(k-1)}n_{2}^{-s'-i}M'^{i}
		      \int\limits_{0}^{\infty}(4 \pi ny)^{k-1} (4\pi y)^{-i}  e^{-4 \pi n y}  \frac{dy}{y} \\
		 & \hspace{1 cm} = \sum\limits_{i=0}^{-s'} (-1)^{i} \binom{-s'}{i}
		      \frac{\Gamma(s'+k-l)} {\Gamma(s'+k-l-i)\Gamma(k-1)}n_{2}^{-s'-i} (M'n)^{i} 
		      \int\limits_{0}^{\infty}y^{k-1}  y^{-i}  e^{- y}  \frac{dy}{y} \\
		 & \hspace{1 cm}= \sum\limits_{i=0}^{-s'} (-1)^{i}  \binom{-s'}{i}
		   \frac{\Gamma(s'+k-l)\Gamma(k-i-1)} {\Gamma(s'+k-l-i)\Gamma(k-1)}n_{2}^{-s'-i} (M'n)^{i} \\
		 & \hspace{1 cm}= P_{s'}(n_{2},M'n).
	\end{align*}
Therefore, for every $n_{1},n_{2}$ such that $C_{f}n_{1}+n_{2}=M'n$,  \eqref{5.24} becomes 
\begin{eqnarray*}
       d(n_{1},n_{2};s',\chi) 
      & = & 2 (C_{f}C_{g})^{-s'} \overline{\chi}(C_{g}) \chi(n_{1}) b(n_{1})
          \sum\limits_{n_{2}=dd'} \psi\overline{ \omega \chi^{2}}(d') d^{2s'+k-l-1} P_{s'}(n_{2},M'n).
\end{eqnarray*}	
	Substituting  the above  expression in \eqref{5.19} finishes the proof.
\end{proof}
\section{Kummer congruences for the distributions}
In this section, we show that the distributions in \eqref{4.6} for $s=l+r$, where $0 \leq r \leq k-l-1$ patch
together into a measure, by verifying the abstract Kummer congruences.

By Proposition~\ref{propn5.3}, with $s = l +r$, where $0 \leq r \leq k-l-1$, we have 
\begin{align}\label{6.1}
        \Psi_{l+r}^{(M')}(\chi) & = \gamma(M') \langle f, f  \rangle_{C_{f}}^{-1} 
         \langle f_{0}^{\rho}\vert V_{C_{f}} , \widetilde{K}_{M'}(r-k+l+1) \vert w_{C_{0}C_{f}} \rangle_{C_{0}C_{f}} 
         \nonumber \\
        &\stackrel{\eqref{2.2}}{=} \gamma(M') \langle f,f \rangle_{C_{f}}^{-1} 
        \langle f_{0}^{\rho}\vert V_{C_{f}}  w_{C_{0}C_{f}}, \widetilde{K}_{M'}(r-k+l+1)  \rangle_{C_{0}C_{f}}. 
\end{align}
By Corollary~\ref{Psi is algebraic} and \eqref{5.12}, we have  $\Psi_{s}^{(M')}(\chi)$ and 
$\gamma(M')$ are algebraic numbers.  Hence, 
\begin{align}\label{6.2}
          \langle f,f \rangle_{C_{f}}^{-1} \langle f_{0}^{\rho}\vert V_{C_{f}} , \widetilde{K}_{M'}(r-k+l-1)
          \vert w_{C_{0}C_{f}}  \rangle_{C_{0}C_{f}} \in \overline{\mathbb{Q}}.
\end{align}
Further, note  that the cusp form $\widetilde{K}_{M'}(r-k+l+1)$ has algebraic Fourier coefficients.  Let 
\begin{align*}
      S_{k}(C_{0}C_{f},\psi;\overline{\mathbb{Q}}) = \lbrace h \in S_{k}(C_{0}C_{f},\psi ) 
      \mid h \mathrm{ \ has \ algebraic \ Fourier \ coeffecients } \rbrace.
\end{align*}
We now claim that $f_{0}^{\rho} \vert V_{C_{f}}w_{C_{0}C_{f}} \in S_{k}(C_{0}C_{f},\psi;\Qbar)$.  
Clearly $f_{0}^{\rho} \vert V_{C_{f}}w_{C_{0}C_{f}}$ belongs to $S_{k}(C_{0}C_{f},\psi)$. 
So it is enough to show that the Fourier coefficients of $f_{0}^{\rho} \vert V_{C_{f}}w_{C_{0}C_{f}}$ 
are algebraic. Observe that
\begin{align*}
      f_{0}^{\rho} \vert V_{C_{f}}w_{C_{0}C_{f}}
     &\stackrel{\mathclap{\eqref{4.1}}}{=}~~ f^{\rho} \vert V_{C_{f}}w_{C_{0}C_{f}} - 
      \alpha'(p)f^{\rho}\vert{V_{p}} V_{C_{f}}w_{C_{0}C_{f}}\\ 
     &=~~C_{f}^{-k/2} f^{\rho}\vert w_{C_{0}} -\alpha'(p) (pC_{f})^{-k/2} f^{\rho}\vert w_{C_{f}C_{g}}
      ~  \mathrm{(from \ the \ definition\ of \ } V_{p},V_{C_{f}},w_{C_{0}C_{f}} )\\
     &\stackrel{\mathclap{\eqref{4.17}}}{=} ~~(pC_{g}C_{f}^{-1})^{k/2} f^{\rho}\vert w_{C_{f}} V_{pC_{g}} -
       \alpha'(p) (pC_{f}C_{g}^{-1})^{-k/2}f^{\rho}\vert w_{C_{f}} V_{C_{g}}\\
    &\stackrel{\mathclap{\eqref{rootnumber}}}{=}~~ (pC_{g}C_{f}^{-1})^{k/2} \Lambda(f^{\rho}) 
     f^{\rho}\vert  V_{pC_{g}} - \alpha'(p) (pC_{f}C_{g}^{-1})^{-k/2} \Lambda(f^{\rho}) f^{\rho}\vert  V_{C_{g}}.
\end{align*}
Since $f$ is primitive, it follows that  $f_{0}^{\rho} \vert V_{C_{f}}w_{C_{0}C_{f}}$ has algebraic 
Fourier coefficients.  Define the linear functional $\mathcal{L}: S_{k}(C_{0}C_{f},\psi) 
\longrightarrow \mathbb{C}$,  by 
\begin{align}\label{L operator}
	\mathcal{L}(K) = \frac{\langle f_{0}^{\rho}\vert V_{C_{f}} w_{C_{0}C_{f}},
	 K  \rangle_{C_{0}C_{f}}}{\langle f,f \rangle_{C_{f}}}.
\end{align}
 We note from \eqref{6.1} and
 \eqref{L operator} that, for every finite order character $\chi : \Z_p^\times \rightarrow \C^\times$,
\begin{equation}\label{6.3}
	\Psi^{(M')}_{l+r} (\chi) = \gamma(M') \mathcal{L}(\widetilde{K}_{M'}( r-k+l+1)).
\end{equation}
\begin{lemma}\label{Lemma 6.1}
	Let $\mathcal{L}$ be defined as above. Then
	\begin{enumerate}[label={\emph{(\arabic*)}}]
		\item $\mathcal{L}$ is defined over $ \overline{\mathbb{Q}}$, i.e., 
		          $\mathcal{L}(S_{k}(C_{0}C_{f},\psi;\overline{\mathbb{Q}})) \subset \overline{\mathbb{Q}}$.
		\item Let $K(z) = \sum_{n=1}^{\infty} a(n,K) e(nz)$ be an element of $S_{k}(C_{0}C_{f},\psi;\Qbar)$.
		         Then there exists $m \in \mathbb{N}$ and $\xi_{1}, \ldots, \xi_{m} \in \Qbar$ such that
		\begin{equation}\label{6.4}
		      \mathcal{L}(K)=  \sum\limits_{n=1}^{m} \xi_{n} a(n,K).
		\end{equation} 
	\end{enumerate} 
\end{lemma}
\begin{proof}
	 Choose an orthogonal basis $f_{1},\ldots, f_{d}$ of $S_{k}(C_{0}C_{f},\psi;\overline{\mathbb{Q}} )$
	  such that $f_{1}=f_{0}^{\rho}\vert V_{C_{f}} w_{C_{0}C_{f}}$. By Proposition~\ref{propn5.3}, we 
	   know that $\widetilde{K}_{M'}(r-k+l+1) \in S_{k}(C_{0}C_{f},\psi;\overline{\mathbb{Q}})$ for all 
	   integers $0 \leq r \leq k-l-1 $. Let $\widetilde{K}_{M'}(r-k+l+1)= \sum_{i=1}^{d} c_{i}f_{i}$, for some 
	   $c_{i} \in \Qbar$. It follows from \eqref{6.2} and orthogonality that 
	$$ 
	     \mathcal{L}(\widetilde{K}_{M'}(r-k+l+1)) = c_{1} \mathcal{L}(f_{0}^{\rho}
	      \vert V_{C_{f}}w_{C_{0}C_{f}}) \in \Qbar.
	 $$
	Choose $r$, $\chi$ such that $\Psi_{l+r}^{(M')}(\chi)= \gamma(M')\mathcal{L} 
	(\widetilde{K}_{M'}(r-k+l+1))  \neq 0$. Such a choice exists, otherwise  all the twisted 
     $L$-values of the Rankin product $L$-function vanish by \eqref{4.8} and 
     Proposition~\ref{Propn 4.1}, so the $p$-adic Rankin product $L$-function, or more 
     precisely the measure $\mu$ in Theorem~\ref{maintheorem}, can be taken to 
     be identically zero.  Hence, $c_1  \neq 0$ and 
     $\mathcal{L}(f_{0}^{\rho} \vert V_{C_{f}}w_{C_{0}C_{f}})  \in \Qbar$. Therefore 
	$\mathcal{L}(S_{k}(C_{0}C_{f},\psi;\overline{\mathbb{Q}}))= \Qbar \mathcal{L}(f_{0}^{\rho} \vert 
	V_{C_{f}}w_{C_{0}C_{f}}) = \Qbar$. This finishes the proof of the first part.

     Let $\mathbb{T}_{k}(C_{0}C_{f},\psi)$ denote the $\Qbar$-subalgebra of 
     End$_{\mathbb{C}}(M_{k}(C_{0}C_{f},\psi))$ generated by the Hecke operators $T_{n}$, for all 
     $n \in \mathbb{N}$. Clearly $\mathbb{T}_{k}(C_{0}C_{f},\psi)$ is a finite dimensional $\Qbar$-vector space.  
    By \cite[Theorem 4.5.13]{Miyake} and \cite[(4.5.27)]{Miyake}  we  obtain  that 
   $\lbrace T_{n} \rbrace_{n\in \mathbb{N}}$ spans $\mathbb{T}_{k}(C_{0}C_{f},\psi)$ as 
   a $\Qbar$-vector space. Hence, by finite dimensionality, there exists $m$ such that $T_{1},\ldots,T_{m}$ span  
   $\mathbb{T}_{k}(C_{0}C_{f},\psi)$ as a $\Qbar$-vector space.
   There is an isomorphism of $\Qbar$-vector spaces given by 
     (see \cite[Lemma 2]{Ghate_congruences})
    \begin{alignat*}{2}
	       \mathbb{T}_{k}(C_{0}C_{f},\psi) & \longrightarrow   \mathrm{Hom}_{\Qbar}
	       (S_{k}(C_{0}C_{f},\psi;\overline{\mathbb{Q}})  ,\Qbar) \\
	        T &\ \mapsto \ a(1,Tf).
    \end{alignat*}
     By the first part of  the lemma  we know that 
     $\mathcal{L} \in \mathrm{Hom}_{\Qbar}(S_{k}(C_{0}C_{f},\psi;\overline{\mathbb{Q}}),\Qbar)$.
     Therefore,  $\mathcal{L}(K)= a(1, TK)$, for some $T \in \mathbb{T}_{k}(C_{0}C_{f},\psi)$. 
     Since, $T_{1}, \cdots , T_{m}$ span $\mathbb{T}_{k}(C_{0}C_{f},\psi)$ as $\Qbar$-vector space, 
     there exists $\xi_{1}, \ldots, \xi_{m} \in \Qbar$ such that, $T = \sum_{n=1}^{m}\xi_{n}T_{n}$. So, 
    $\mathcal{L}(K)= \sum_{n=1}^{m}\xi_{n} a(1,T_{n}K)= \sum_{n=1}^{m}\xi_{n} a(n,K), \forall \ K 
     \in S_{k}(C_{0}C_{f},\psi;\overline{\mathbb{Q}})$.
\end{proof}

As mentioned earlier, every complex-valued Dirichlet character $\chi$ on 
$\mathbb{Z}_p^\times$ takes values in $\Qbar \subset \C$.  
From now on we think of such character as taking values in $\C_p$ via our fixed  embedding 
$i_{p}:\overline{\mathbb{Q}} \rightarrow \mathbb{C}_{p}$.
Since $\Psi_{l+r}(\chi) \in \Qbar$, for $0 \leq r \leq k-l-1$, by Corollary~\ref{Psi is algebraic}, 
we have $i_p(\Psi_{l+r}(\chi)) \in \C_p$.
Thus we may think of the complex distribution $\Psi_{l+r}$, as a 
$\mathbb{C}_{p}$-valued distribution. We shall denote these distributions by 
$i_p(\Psi_{l+r})$, for $0 \leq r \leq k-l-1$. 
We now define a candidate for the measure in Theorem~\ref{maintheorem}, namely we take
\begin{eqnarray} \label{mu}
      \mu := i_{p}(\Psi_{l}).
\end{eqnarray}
By Proposition~\ref{Propn 4.1} and\eqref{6.3} we have 
\begin{eqnarray}\label{Psi and L}
	   \Psi_{l+r} (\chi)
	    & = & \gamma(M') \mathcal{L}(\widetilde{K}_{M'}( r-k+l+1)) \nonumber \\
         & = & \gamma(M') \sum_{n=1}^m \xi_n \> a(n, (\widetilde{K}_{M'}( r-k+l+1)) 
                    \quad  \text{(by Lemma~\ref{Lemma 6.1} (2))} \nonumber \\
          & = & \gamma(M') \sum_{n = 1}^m \xi_n  \sum\limits_{C_{f}n_{1}+n_{2} = M'n} 
                    d(n_{1},n_{2};r-k+l+1,\chi),  
\end{eqnarray}
by Proposition~\ref{propn5.3}, where $M'$ is a sufficiently large power of $p$ chosen
 depending on $\chi$, and $\gamma(M')$ is as defined in \eqref{5.12}.  As remarked
earlier, $\gamma(M')$ is $p$-integral in many cases (apparently in all), but in any 
case has bounded denominator, coming from $\Lambda(g)$, since $\alpha(pM')$ is a 
$p$-adic unit. Similarly, the $\xi_n \in \Qbar$ have bounded denominators. 
Finally the $d(n_{1},n_{2};s,\chi)$ also have denominators at worst $\Gamma(k-1)$ 
by \eqref{5.16}, \eqref{5.17}. Hence multiplying $i_p(\Psi_{l+r})$ by a suitable (fixed) 
power of $p$ we may and do assume that $i_p(\Psi_{l+r}(\chi))$ lies in $\mathcal{O}_p$
for all $\chi$. Proving that this rescaled distribution is an $\mathcal{O}_p$-valued measure 
will imply that $i_p(\Psi_{l+r})$ is a (not necessarily  $\mathcal{O}_p$-valued) measure.
\begin{proposition}\label{prop Kummer}  
       For all integers  $ 0 \leq r \leq k-l-1$, we have 
	   \begin{enumerate}[label={\emph{(\arabic*)}}]
	           \item{The $\mathbb{C}_{p}$-valued distributions $i_{p}(\Psi_{l+r})$  are bounded. 
	                     Hence, $i_{p}(\Psi_{l+r})$ are measures on $\mathbb{Z}_{p}^{\times}$.}
	          \item{Moreover, with $\mu$ as in \eqref{mu}, the following equality holds
	                    \footnote{The formula \eqref{measure sign} differs from \cite[(5.6)]{Panrankin} by 
                                          the factor $(-1)^r$.  This factor is forced on us in view of the sign corrections 
                                          mentioned in the previous footnotes. Moreover, this sign has 
                                          theoretical significance: \eqref{measure sign} matches with a general 
                                          expectation about measures attached to $L$-functions of motives \cite[(4.16)]{CP89}. } 
	          \begin{equation}\label{measure sign}
	 	             \int\limits_{\mathbb{Z}_{p}^{\times}} \chi x_{p}^{r} \ d\mu =
	 	             (-1)^{r}\int\limits_{\mathbb{Z}_{p}^{\times}} \chi di_{p}(\Psi_{l+r}).
              \end{equation} }
	  \end{enumerate}
\end{proposition}
\begin{proof} 
	  Fix an integer $0 \leq r \leq k-l-1$. Recall that the linear span of  $\mathcal{B} = \lbrace \chi \mid \chi: 
	  \mathbb{Z}_{p} ^{\times} \rightarrow \mathbb{C}_{p}^{\times} \mathrm{ \ has \ finite \ order }
	  \rbrace $ is dense in ${\mathcal{C}}(\mathbb{Z}_{p}^{\times}, \mathbb{C}_{p})$. 
      We claim that the distribution $i_{p}(\Psi_{l+r})$ satisfies the abstract Kummer  
      congruences \eqref{3.4} with  $\mathcal{B}$ as the system of functions. We need to prove  
      that for  every  finite  set  of  characters $ \chi_{1}, \ldots, \chi_{t} \in \mathcal{B}$, constants
      $c_{1}, \cdots , c_{t} \in \mathbb{C}_{p}$ and $m \geq 0$,
     \begin{align*}
	              \mathrm{if} \ \sum \limits_{i=1}^{t} c_{i} \chi_{i} \equiv 0 \ (\mathrm{mod} \ p^{m}),  
                   ~ \mathrm{then} \ \sum\limits_{i=1}^{t} c_{i} \> i_{p}(\Psi_{l+r}(\chi_{i})) \equiv 0 
                   ~(\mathrm{mod} \ p^{m}).
      \end{align*}
           Choose  $M'$ sufficiently large  so that \eqref{Psi and L} holds for each of the $\chi_i$. 
           By \eqref{Psi and L}, this is equivalent to proving
            \begin{align}\label{6.7}
	                \mathrm{if} \ \sum \limits_{i=1}^{t} c_{i} \chi_{i} \equiv 0 \ (\mathrm{mod} \ p^{m}),  
                    ~ \mathrm{then} \ \sum\limits_{i=1}^{t} c_{i} \> d(n_{1},n_{2};r-k+l+1,\chi)  
                     \equiv 0 \ (\mathrm{mod} \ p^{m}).
           \end{align}
         for each $n$, and each $n_{1}$, $n_{2}$ satisfying $C_{f}n_{1}+n_{2}=M'n$.

          If $p\mid n_{1}$, then $d(n_{1},n_{2};r-k+l+1,\chi_{i})=0$, by \eqref{5.16}, since 
          $\chi_i(n_1) = 0$.  So, the relation \eqref{6.7} is trivially true.
          Hereafter, we  assume $p \nmid n_{1}$.  Since $p\nmid C_{f}$, $p \mid M'$ and 
          $C_{f}n_{1}+n_{2}=M'n$,  we have  $p \mid n_{1}$ if and only if $p \mid n_{2}$. 
          So we have $p \nmid n_2$ and $d/d'C$ is a $p$-adic unit, for $dd' = n_2$.
		 Let $s=r-k+l+1$. We may also assume $M'$ has been chosen large enough so that 
		 $p^{m} \mid M'/\Gamma(k-1)$. From  \eqref{5.17} and the equality 
		 $C_{f}n_{1}+n_{2}=M'n$, it follows that 
		\begin{equation}\label{6.8}
		\begin{split}
		P_{s}(n_{2},M'n) & \equiv  n_{2}^{k-l-1-r} \equiv  (dd')^{k-l-1-r} \ \ (\mathrm{mod} \ p^{m}), \\
		 \chi(n_{1}) &= \overline{\chi}(-C_{f})\chi(n_{2})= \overline{\chi}(-C_{f})\chi(dd').
		 \end{split}
		\end{equation} 
         By  \eqref{5.16} and  \eqref{6.8}, we have the congruence
		\begin{eqnarray*}
                \label{congruenceford}
		        d(n_{1},n_{2};r-k+l+1,\chi) & \equiv &  2   \overline{\chi}(-C) C^{k-l-r-1}b(n_{1})
		         \sum\limits_{n_{2}=dd'} \psi\overline{\omega \chi} (d') 
                   \chi(d) (d')^{k-l-r-1}  d^{r} \nonumber \\  
                   & & \quad \qquad  \quad \qquad \qquad \qquad \qquad \qquad \qquad (\mathrm{mod} \ p^{m}). 
		\end{eqnarray*}
		Therefore,
		\begin{align}\label{6.6}	
		(-1)^{r} \sum\limits_{i=1}^{t}  c_{i} \> d(n_{1},n_{2};r-k+l+1,\chi_{i}) \equiv 
		2  b(n_{1})  \sum\limits_{n_{2}=dd'} \psi\overline{\omega} (d')   (d'C)^{k-l-1}  
		& \sum\limits_{i=1}^{t} c_{i}   \chi_{i} \Big (\frac{-d}{d'C} \Big ) \Big ( \frac{-d}{d'C} \Big )^{r} \nonumber \\
		& (\mathrm{mod} \ p^{m}).
		\end{align}
		By assumption,  $\sum_{i} c_{i}\chi_{i} \equiv 0 $ (mod $p^{m}$), so 
		$\sum_{i}^{} c_{i} \chi_{i}(-d/d'C) \equiv 0 \ (\mathrm{ mod} \ p^{m})$. 
         Since each $-d/d'C$ is a $p$-adic unit, we obtain
		 $\sum_{i=1}^{t} c_{i} \> d(n_{1},n_{2};r-k+l+1,\chi_{i}) \equiv 0 \ (\mathrm{mod} \ p^{m})$.
		Thus \eqref{6.7} holds and this finishes the proof of (1). 

        For (2), we claim there exists a $\mathbb{C}_{p}$-valued measure $\nu$ such that 
		\begin{align*}
		       \int\limits_{\mathbb{Z}_{p}^{\times}} \chi x_{p}^{r} \ d\nu =(-1)^{r} 
		        \int\limits_{\mathbb{Z}_{p}^{\times}} \chi \ di_{p}(\Psi_{l+r}), ~~~~ \forall  
		        ~ ~ 0 \leq r \leq k-l-1.
		\end{align*}
		Let $\mathcal{B}' = \lbrace \chi x_{p}^{r} \mid \chi \in \mathcal{B} ~
		 \mathrm{and} \ 0 \leq r \leq k-l-1 \rbrace $. 
         To prove the existence of this measure, it is enough to  verify the abstract 
          Kummer congruences hold for $\mathcal{B}'$ as the system of functions. 
         As in (1), we need to prove for every  finite  set  of  characters $ \chi_{i} \in \mathcal{B}'$ 
         and $c_{i,r} \in \mathbb{C}_{p}$, 
		\begin{equation}\label{6.9}
		       \mathrm{if} \ \sum \limits_{i,r} c_{i,r} \chi_{i} x_{p}^{r} \equiv 0 ~
		         (\mathrm{mod} \ p^{m}), \ \mathrm{then} \ 
		         \sum\limits_{i,r} (-1)^{r} c_{i,r} \> d(n_{1},n_{2};r-k+l+1,\chi_{i}) 
		         \equiv 0 \ (\mathrm{mod} \ p^{m}).
		\end{equation}
		As observed above, if $p \mid n_{1}$, then $d(n_{1},n_{2};r-k+l+1,\chi_{i})=0$, so \eqref{6.9} holds. 
         For $p\nmid n_{1}$, it follows from  \eqref{6.6}  that 
		\begin{equation*}
		      \begin{split}
		              \sum\limits_{i,r} (-1)^{r} c_{i,r} \> d(n_{1},n_{2};r-k+l+1,\chi_{i}) \equiv   
		              \sum\limits_{n_{2}=dd'} 2 & b(n_{1})  \psi\overline{\omega} (d') (d'C)^{k-l-1} \\ 
		              & \times \bigg (  \sum\limits_{i,r} c_{i,r} \chi_{i} \Big ( \frac{-d}{d'C}  \Big ) 
		              \Big ( \frac{-d}{d'C} \Big )^{r} \bigg ) \ \ (\mathrm{mod} \ p^{m}).
		       \end{split}
		\end{equation*}
		By the assumption in \eqref{6.9},  the inner sum is congruent to 0 (mod $p^{m})$. Thus
		\begin{align*}
		      \sum\limits_{i,r} (-1)^{r} c_{i,r} \> d(n_{1},n_{2};r-k+l+1,\chi_{i}) \equiv 0 
		       \ (\mathrm{mod} \ p^{m}),
		\end{align*}
                so again \eqref{6.9} holds.
		This proves that $\nu$ as claimed above exists. Further $\mu$ and $ \nu$ agree on
		 $\mathcal{B}$ (take $r = 0$) which spans Step$(\mathbb{Z}_{p}^{\times}, \mathbb{C}_{p})$.
		 Hence, $\mu = \nu$. This completes the proof of (2).
 \end{proof}

Let $\mu$ be the distribution in \eqref{mu}. By Proposition~\ref{prop Kummer} (1) 
with $r =0$, we see that $\mu$ is a measure. By \eqref{measure sign}, and \eqref{4.8} 
with $s = l + r$, we see that $\mu$ satisfies the interpolation property
\footnote{The sign  $(-1)^r$ in \eqref{measure sign} directly contributes to the 
corrected sign $(-1)^r$ in Theorem~\ref{maintheorem}.} of Theorem~\ref{maintheorem}. 
This completes the proof of Theorem~\ref{maintheorem}.

\vspace{.2cm}

{\noindent \bf  Acknowledgements:} We thank B. Balasubramanyam, D. Benois, D. Loeffler
and S. Kobayashi for helpful conversations.

\begin{flushleft}

\end{flushleft}

{\noindent Address:} School of Mathematics, TIFR, Homi Bhabha Road, Mumbai 400005, India

{\noindent Email:} {\tt{eghate@math.tifr.res.in, ravithej@math.tifr.res.in}}

\end{document}